\documentclass[a4paper,11pt]{article}

\usepackage{scalerel}[2016/12/29]
\usepackage{color}
\usepackage[latin1]{inputenc}
\usepackage[T1]{fontenc}
\usepackage[normalem]{ulem}
\usepackage[english]{babel}
\usepackage{verbatim}
\usepackage{graphicx}
\usepackage{enumitem}
\usepackage{amsmath,amssymb,amsfonts,amsthm,mathrsfs}
\usepackage{rotating,relsize}
\usepackage{float}
\usepackage{array}
\usepackage{todonotes}
\usepackage{MnSymbol}
\usepackage[all,cmtip]{xy}
\usepackage{tikz}
\usepackage{tikz-cd}
\usepackage{adjustbox}
\usepackage[hidelinks]{hyperref}
\usepackage[clock]{ifsym}
\usepackage{stackengine}
\usepackage{scalerel}
\usepackage{stmaryrd}

\stackMath
\xyoption{all}

\newtheorem{theorem}{Theorem}[section]

\newtheorem*{theorem*}{Theorem}
\newtheorem*{cor*}{Corollary}

\newtheorem{lemma}[theorem]{Lemma}
\newtheorem{prop}[theorem]{Proposition}
\newtheorem{cor}[theorem]{Corollary}

\theoremstyle{definition}
\newtheorem*{remark*}{Remark}
\newtheorem{defn}[theorem]{Definition}
\newtheorem{rem}[theorem]{Remark}
\newtheorem{example}[theorem]{Example}

\DeclareMathOperator{\id}{id}

\DeclareMathOperator{\tr}{tr}

\DeclareMathOperator*{\colim}{colim}

\DeclareMathOperator{\THH}{THH}

\DeclareMathOperator{\TR}{TR}

\DeclareMathOperator{\ev}{ev}

\DeclareMathOperator{\Hom}{Hom}

\DeclareMathOperator{\Set}{Set}

\DeclareMathOperator{\rat}{rat}

\DeclareMathOperator{\N}{\mathbb{N}}
\DeclareMathOperator{\Z}{\mathbb{Z}}
\DeclareMathOperator{\F}{\mathbb{F}}
\DeclareMathOperator{\Q}{\mathbb{Q}}

\DeclareMathOperator{\bimod}{biMod}

\DeclareMathOperator{\tlog}{tlog}
\DeclareMathOperator{\End}{End}
\DeclareMathOperator{\SEnd}{SEnd}
\DeclareMathOperator{\Mat}{Mat}
\DeclareMathOperator{\Proj}{Proj}
\newcommand{\cyctens}[3]{{#2}\lower.6ex\hbox{${}^{\underset{\scaleto{ #1}{3pt}}{\circledcirc} p^{#3}}$}}
\newcommand{\cycsmash}[3]{{#2}\lower.6ex\hbox{${}^{\underset{\scaleto{ #1}{3pt}}{\varowedge} p^{#3}}$}}
\newcommand{\tens}[3]{{#2}\lower.6ex\hbox{${}^{\underset{\scaleto{ #1}{3pt}}{\otimes} p^{#3}}$}}

\newcommand{\pow}[1]{\left[\kern-0.5ex\left[{#1}\right]\kern-0.5ex\right]}
\newcommand{\xto}{\xrightarrow}

\newtheorem*{thma}{Theorem A}
\newtheorem*{thmb}{Theorem B}

\begin{document}
\begin{center}\LARGE{Witt vectors with coefficients and characteristic polynomials over non-commutative rings}
\end{center}

\begin{center}\Large{Emanuele Dotto, Achim Krause, Thomas Nikolaus, Irakli Patchkoria}
\end{center}

\vspace{.05cm}

\abstract{
For a not-necessarily commutative  ring $R$ we define an abelian group $W(R;M)$ of Witt vectors with coefficients in an $R$-bimodule $M$. These groups generalize the usual big Witt vectors of commutative rings and we prove that they have analogous formal properties and structure. One main result is that $W(R) := W(R;R)$ is Morita invariant in $R$.   

For an $R$-linear endomorphism $f$ of a finitely generated projective $R$-module we define a characteristic element $\chi_f \in W(R)$. This element is a non-commutative analogue of the classical characteristic polynomial and we show that it has similar properties. The assignment $f \mapsto \chi_f$ induces an isomorphism between a suitable completion of cyclic $K$-theory $K_0^{\mathrm{cyc}}(R)$ and $W(R)$. 
}

\vspace{.05cm}


\section*{Introduction}
\phantomsection\addcontentsline{toc}{section}{Introduction}

In this paper we define and study big Witt vectors with coefficients: concretely for a not-necessarily commutative ring $R$ and an $R$-bimodule $M$ we will define an abelian group
$W(R;M)$ called the group of \emph{big Witt vectors} of $R$ with coefficients in $M$. We will start by focusing on the case $M = R$ and set $W(R) := W(R;R)$. 

\begin{itemize}
\item If $R$ is a commutative ring then our group $W(R)$ is the underlying group of the classical ring of big Witt vectors. The latter is a `global' variant of the rings of   $p$-typical Witt vectors.

\item If $R$ is non-commutative then our group $W(R)$ agrees with the big non-commutative Witt vectors introduced by the second and third author in \cite{THHBook} as a global variant of Hesselholt's non-commutative $p$-typical Witt vectors \cite{HesselholtncW, HesselholtncWcorr}. We note that  $W(R)$ is in general only an abelian group, but has an `external product'  $W(R) \otimes W(R) \to W(R \otimes R)$ generalizing the ring structure in the commutative case. 
\end{itemize}
The abelian group $W(R)$ is defined as 
\begin{equation}\label{def_witt}
W(R) := \frac{\left(1 + tR\pow{t}\right)^{\mathrm{ab}}}{ 1 + rst \sim 1 +srt} 
\end{equation}
where $1 + tR\pow{t}$ is the multiplicative group of power series with
constant term 1.\footnote{Here the abelianisation as well as the quotient are taken in separated topological groups with respect to the $t$-adic topology. Concretely that amounts to quotienting by the closures of the subgroups generated by the imposed relations.}  One of our main results is that $W(R)$ is invariant under Morita equivalence in $R$, and we will see that our proof crucially uses the variant of Witt vectors with coefficients. \\

One of our motivations to study these groups is to define characteristic polynomials for endomorphisms over non-commutative rings. Recall that if $R$ is commutative and $A$ is an $(n\times n)$-matrix over $R$ then we have the (inverse) characteristic polynomial
\begin{equation}\label{char_el_eqn}
\chi_A(t) = \det(\id - At) 
\end{equation}
which can be considered as an element in the abelian group $W(R) = 1 + tR\pow{t}$. It has the following properties:
\begin{enumerate}[label=(\roman*)]
\item
\label{eins}
It satisfies the trace property $\chi_{AB} = \chi_{BA}$. In particular $\chi_{SAS^{-1}} = \chi_{A}$ so that it is independent of the choice of basis;
\item
For a matrix of the form
$
A = 
\begin{pmatrix}
  A_1 & * \\
  0 & A_2 
\end{pmatrix}
$
we have $\chi_{A} = \chi_{A_1} \cdot \chi_{A_2}$ and $\chi_{0_n} = 1$;
\item
The negative of the logarithmic derivative
is given by 
\[
- \frac{\chi'_A(t)}{\chi_A(t)} =  \tr(A) + \tr(A^2) t + \tr(A^3) t^2 + ...  \ ;
\]
\item\label{five} The polynomial $\chi_A$ is natural in $R$.
\end{enumerate}

In \S\ref{Section: char polynomial} we generalize $\chi_A$ in two directions: we allow $R$ to be non-commutative and  we replace the matrices $A$ by $R$-linear endomorphisms $f: P \to P$ of arbitrary finitely generated, projective $R$-modules $P$.\footnote{Working out the definition and properties of $\chi_f$ for such endomorphisms $f: P \to P$ over commutative rings $R$ is a nice little exercise for the reader. }

\begin{thma}\label{thm_main}
For every endomorphism $f: P \to P$ of a finitely generated, projective $R$-module $P$ there is an element $\chi_f \in W(R)$ generalizing the inverse characteristic polynomial \eqref{char_el_eqn} and which satisfies the analogues of properties \ref{eins}-\ref{five} above. 
\end{thma}

We define $\chi_f$ by an appropriate version of  formula \eqref{char_el_eqn}  using a non-commutative variant of the determinant (which we also construct). Before we explain this strategy in more detail, let us note that 
an immediate corollary of Theorem A is that the assignment $f \mapsto \chi_f$ defines a map
\[
K_0^{\mathrm{cyc}}(R) \to W(R)
\]
where $K_0^{\mathrm{cyc}}(R)$ is the zero'th cyclic $K$-theory group of $R$ (see Definition \ref{cyclictrace}). 
Such a map was previously constructed using homotopy theoretic methods, notably the cyclotomic trace, and our main motivation was to give a purely algebraic description of this map. \\

In order to prove Theorem A, i.e. to define $\chi_f$, the Morita invariance of non-commutative Witt vectors is used in an essential way: the polynomial $(\id - ft )$ can naturally be considered as an element of $W(\End_R(P))$. By Morita invariance we have a canonical map 
\begin{equation}\label{map_one}
W(\End_R(P)) \to W(R)
\end{equation}
so that we simply define $\chi_f$ as the image of  $(\id - ft )$ under the map \eqref{map_one}. 
The map \eqref{map_one} in turn is a special case of the fact that for every additive functor $\Proj_S \to \Proj_R$ between categories of finitely generated, projective modules over rings $S$ and $R$, we get an induced map $W(S) \to W(R)$ on Witt vectors. 
Given the definition of $W(R)$ this is highly non-obvious: 
the idea is to first introduce groups $W(R; M)$ of Witt vectors with coefficients in a bimodule $M$ by 
replacing the power series ring in \eqref{def_witt} by the completed tensor algebra of $M$ over $R$. Then the main result, which we prove in \S\ref{sec:Morita}, is that this construction satisfies the \emph{trace property} (here we use terminology from Kaledin inspired by work of  Ponto):
 \begin{thmb}
 For an $S$-$R$-bimodule $M$ and an $R$-$S$-bimodule $N$ there is an isomorphism
 \[
 W(S; M \otimes_R N) \cong W(R;  N \otimes_S M) \ .
 \]
 \end{thmb}
Using this result and the fact that every additive functor $\Proj_S \to \Proj_R$ is of the form $-\otimes_S M$ one formally gets an induced map $W(S) \to W(R)$, see Corollary \ref{cor:morita}.

Besides the trace property, we also generalize the structures present on classical Witt vectors of commutative rings, such as multiplication, Frobenius and Verschiebung maps, to the groups $W(R; M)$. The analogues of those structures in our setting are `external', for example the $p$-th Frobenius $F_p$ is a map
$W(R; M) \to W(R; M^{\otimes_R p})$ (see \S\ref{sec:operators}). We also define a ghost component map which is essentially given by the logarithmic derivative (see Proposition \ref{prop:tlogspecialvalues}) as well as $p$-typical and truncated Witt vectors with coefficients for non-commutative rings (see \S\ref{sec:truncated}). \\

We note that characteristic polynomials (and determinants) for non-commutative rings have been considered before by Ranicki \cite{Ranicki} and Sheiham \cite{Sheiham, Sheiham2}. We reformulate their approach and compare it to ours in \S\ref{sec_ratWitt}. Let us quickly summarize the situation: 
for commutative rings $R$ the characteristic polynomial $\chi_f$ is a polynomial rather than a power series in $W(R) = 1 + tR\pow{t}$.  The subgroup  $W^{\rat}(R) \subseteq W(R)$ generated by polynomials is called the group of rational Witt vectors (and it is a subring). Then the fact that $\chi_f$ is a polynomial shows that the element $\chi_f$ as well as the image of $K_0^{\mathrm{cyc}}(R) \to W(R)$ lie in this subgroup. In the non-commutative situation this unfortunately turns out to be false: in general  $\chi_f \in W(R)$ can not be represented by a polynomial! 

However, we can still define a group $W^{\rat}(R)$ of rational Witt vectors for non-commutative rings (Definition \ref{def_ratWitt}) together with a not-necessarily injective homomorphism $W^{\rat}(R) \to W(R)$ and a lift $\chi_f^{\rat}$ of $\chi_f$. In fact this map is a completion  and the assignment $f \mapsto \chi^{\rat}_f$ defines an isomorphism between the groups $K_0^{\mathrm{cyc}}(R)$ and $W^{\rat}(R)$ as shown by Sheiham, generalizing earlier work of Almkvist. Unfortunately, in order to establish the existence and the properties of the group $W^{\rat}(R)$ as well as the element $\chi^{\rat}_f$ one crucially uses cyclic $K$-theory and a version of the Gauss algorithm. 
We have not been able to give a satisfactory, self-contained treatment of $W^{\rat}(R)$ and $\chi^{\rat}_f$ similar to our treatment of $W(R)$ and $\chi_f$ (see Remark \ref{rem_rat_trace}). 

\subsection*{Relation to other work} 

As indicated before, our definition of $W(R;M)$ was inspired by topological constructions. We will prove the precise connection in a forthcoming companion paper \cite{WittII}. More precisely we will show that there is a natural isomorphism
\begin{equation}\label{isoTR}
W(R;M) \cong \pi_0\TR(R;M).
\end{equation}
Here the spectrum $\TR(R;M)$ was defined by Lindenstrauss-McCarthy in \cite{LMcC} using topological Hochschild homology $\THH(R;M)$ with its `external' cyclotomic structure. For $M = R$ the spectrum $\TR(R;R)$ is the spectrum $\TR(R)$ studied by Hesselholt-Madsen, and in this case our isomorphism \eqref{isoTR} recovers and generalizes their results  \cite[3.3]{Wittvect} as well as the non-commutative analogue of Hesselholt \cite{HesselholtncW, HesselholtncWcorr}. A special case of the isomorphism \eqref{isoTR} lets us compute $\pi_0$ of the  Hill--Hopkins--Ravenel norm \cite{HHR} for cyclic groups. For example, we get for any connective spectrum $X$ an isomorphism
\[
\pi^{C_{p^n}}_0\left(  N_{e}^{C_{p^n}}X \right) \cong W_{p,n}(\mathbb{Z}; \pi_0 X) .
\]
Finally, Kaledin defines in  \cite{KaledinWpoly} (see also \cite{KaledinncW}) abelian groups $\widetilde{W}_n(V)$ of `polynomial Witt vectors' for a vector space $V$ over a  perfect field $k$ of characteristic $p$. We will also show in the forthcoming paper \cite{WittII} that his group $\widetilde{W}_n(V)$ is isomorphic to our group $W_{p,n}(k; V)$ of truncated, p-typical Witt vectors with coefficients.

\newpage
\tableofcontents

\subsection*{Acknowledgements}

The whole approach of the current paper (especially our use of the trace property) is influenced by the work of Kaledin on the subject. We would like to thank him for writing these inspiring papers.  We would also like to thank Lars Hesselholt for the idea of defining non-commutative Witt vectors and many helpful discussions and questions. Finally, we are grateful to Christopher Deninger for helpful comments on an earlier version of this paper.

The first and the fourth authors were supported by the German Research Foundation Schwer\-punktprogramm 1786 and by the Hausdorff Center for Mathematics at the University of Bonn.
The second and third authors were funded by the Deutsche Forschungsgemeinschaft (DFG, German Research Foundation) under Germany's Excellence Strategy EXC 2044 390685587, Mathematics M\"unster: Dynamics--Geometry--Structure.

\section{Big Witt vectors with coefficients}

In this section we define for any pair of a ring $R$ and a bimodule $M$ an abelian group $W(R;M)$ of Witt vectors of $R$ with coefficients of $M$. For a commutative ring $R$ and $M=R$ the group $W(R;R)$ recovers the usual additive group of (big) Witt vectors. For general $R$, $W(R;R)$ therefore forms a noncommutative analogue of Witt vectors, which was in the $p$-typical case first considered by Hesselholt \cite{HesselholtncW}. Like their commutative counterpart, our Witt vectors with coefficients carry additional structure, namely Verschiebung and Frobenius maps, which interact with coefficients in an interesting way, as well as an ``external'' multiplication map. But there is also additional structure which is not seen in the classical picture, namely a residual $C_n$-action if we take coefficients of the form $M^{\otimes_R n}$, and, more generally, \emph{trace property isomorphisms} $W(R;M\otimes_S N)\xto{\cong} W(S; N\otimes_R M)$. These imply that $W(R;R)$ is Morita invariant in $R$.

\subsection{Preliminaries: The category of bimodules}\label{secprelim}

We will consider the category $\bimod$ of pairs $(R;M)$ where $R$ is a ring (unital, associative, but not necessarily commutative) and $M$ is an $R$-bimodule. A morphism $(R;M)\to (R';M')$ is a pair $(\alpha;f)$ where $\alpha : R\to R'$ is a ring homomorphism and $f : M\to \alpha^{*}M'$ is a map of $R$-bimodules, where $\alpha^{*}$ is the restriction of scalars. We will often denote a morphism only by $f$ and keep $\alpha$ implicit.

Given a bimodule $(R;M)$ and an integer $n\geq 1$, we define an $R$-bimodule $M^{\otimes_Rn}$ and an abelian group $M^{\circledcirc_Rn}$ respectively by
\[
M^{\otimes_Rn}=\underbrace{M\otimes_R M\otimes_R\dots\otimes_R M}_{n} \ \ \ \ \ \  \ \ \mbox{and} \ \ \ \  \ \ \ \  M^{\circledcirc_Rn}=M^{\otimes_Rn}/[R,M^{\otimes_Rn}]
\vspace{-.3cm}
\]
where $[R,M^{\otimes_Rn}]$ is the abelian subgroup generated by the elements $rm-mr$ for $r\in R$ and $m\in M^{\otimes_Rn}$. We think of $M^{\circledcirc_Rn}$ as $n$ copies of $M$ tensored together around a circle, and these have a natural action of the cyclic group $C_{n}$ where a chosen generator $\sigma \in C_n$ acts by
\[
\sigma(m_1\otimes\dots\otimes m_{n-1}\otimes m_n):=m_n\otimes m_1\otimes\dots\otimes m_{n-1}.
\]

\begin{example}
When $n=1$ we have that $M^{\circledcirc_R 1}=M/[R,M]$. When $M=R$ there is a canonical isomorphism $R^{\circledcirc_R n}\cong R/[R,R]$ with the quotient by the additive subgroup of commutators, for all $n\geq 1$. If $R$ is commutative and $M$ is an $R$-module considered as a bimodule, then $M^{\circledcirc_R n}\cong M^{\otimes_R n}$.
\end{example}

\begin{defn}\label{freeres}
A bimodule $(R;M)$ is called free if $R$ is a free ring and $M$ is a free $R$-bimodule. A free resolution of $(R;M)$ is a reflexive coequalizer
\[
\xymatrix{(R_1;M_1)\ar@<.7ex>[r]\ar@<-.7ex>[r]&(R_0;M_0)\ar[l]\ar@{->>}[r]&(R;M)}
\]
in the category of bimodules, where $(R_0;M_0)$ and $(R_1;M_1)$ are free.
\end{defn}

\begin{rem}
\label{rem:freeres}
It turns out that reflexive coequalizers in $\bimod$ are computed on underlying sets. That is, $(R;M)$ is a reflexive coequalizer as in Definition \ref{freeres} if and only if the underlying diagrams
\[
\xymatrix{R_1\ar@<.7ex>[r]\ar@<-.7ex>[r]&R_0\ar[l]\ar@{->>}[r]&R}\ \ \ \ \ \ \ \mbox{and}  \ \ \ \ \ \ \ \ \xymatrix{M_1\ar@<.7ex>[r]\ar@<-.7ex>[r]&M_0\ar[l]\ar@{->>}[r]&M}
\]
are reflexive coequalizers of sets (or equivalently of abelian groups, or for the first one of rings). 
To see this, observe that the category $\bimod$ is equivalent to algebras of an operad with two colours (one for the ring, one for the bimodule) in abelian groups.  Thus sifted colimits are computed on underlying pairs of abelian groups. Finally the forgetful functor from abelian groups to sets commutes with sifted colimits.

It follows that any object $(R;M)$ of $\bimod$ admits a free resolution, that can be constructed by taking $R_0=\Z\{R\}$ and $R_1=\Z\{\Z\{R\}\}$ to be the free rings respectively on the underlying sets of $R$ and $\Z\{R\}$, $M_0$ the free $R_0$-bimodule on the underlying set of $M$, and $M_1$ the free $R_1$-bimodule on the underlying set of $M_0$. This is the canonical resolution associated to the adjoint pair
\[
\xymatrix{U : \bimod \ar@<.5ex>[r]&\Set\times \Set\ar@<.5ex>[l] : F}
\]
where $U$ sends $(R;M)$ to the pair of underlying sets $(R;M)$, and $F(X,Y)=(\Z\{X\};\Z\{X\}^e(Y))$. The associated diagram
\[
  \xymatrix{FUFU(R;M)\ar@<.7ex>[r]\ar@<-.7ex>[r]&FU(R;M)\ar[l]\ar@{->>}[r]&(R;M)} 
\]
  exhibits $(R;M)$ as reflexive coequalizer, since this can be computed on underlying pairs in $\Set\times \Set$, where the diagram becomes split by the unit of the adjunction.
\end{rem}

\begin{lemma}\label{restorfree}
For a free bimodule $(S;Q)$, the groups $(Q^{\circledcirc_{S}n})^{C_n}$ and $(Q^{\circledcirc_{S} n})_{C_n}$ are torsion free. In particular any bimodule $(R;M)$ can be resolved by $(S;Q)$ and $(S';Q')$ with torsion-free $(Q^{\circledcirc_{S} n})_{C_n}$ , $(Q'^{\circledcirc_{S'} n})_{C_n}$, $(Q^{\circledcirc_{S} n})^{C_n}$ and $(Q'^{\circledcirc_{S'} n})^{C_n}$.

\end{lemma}
\begin{proof}
Say $S$ is a free ring on the set $X$ of generators, and $Q$ is the free $S$-bimodule on the set $Y$ of generators, i.e. $\bigoplus_{Y} S\otimes_{\Z} S$. Then it is easily seen that $Q^{\circledcirc_{S}n}$ is a direct sum $\bigoplus_{Y^{\times n}} (S\otimes_{\Z} S)^{\circledcirc_{S} n}$, where $C_n$ acts on the index set $Y^{\times n}$ by permuting the factors cyclically, and on the summands by the $C_n$ action on the cyclic tensor product. The cyclic tensor product $(S\otimes_{\Z} S)^{\circledcirc_{S} n}$ is equivalent to $S^{\otimes_{\Z} n}$ with $C_n$ acting by cyclic permutation.

As an abelian group $S$ is free on a set $T$, and $S^{\otimes_{\Z} n}$ is free abelian on the set $T^{\times n}$, with $C_n$ acting by permutation. Thus the whole $Q^{\circledcirc_{S}n}$ is a free abelian group on the set $Y^{\times n} \times T^{\times n}$, with $C_n$ acting by cyclic permutation on both factors. So the $C_n$-invariants are torsion free, because they are a subgroup, and the coinvariants are the free abelian group on the set $(Y^{\times n} \times T^{\times n})/ C_n$, thus also torsion free.
\end{proof}

The category $\bimod$ has a monoidal structure, which is defined by the componentwise tensor product
\[
(R;M)\otimes (R';M'):=(R\otimes_{\Z}R';M\otimes_{\Z}M')
\]
where $M\otimes_{\Z}M'$ has the obvious $R\otimes_{\Z}R'$-bimodule structure.

\begin{lemma}\label{monbimod}
 The category of monoids in $\bimod$ is isomorphic to the category of pairs $(R;M)$ where $R$ is a commutative ring and $M$ is a ring equipped with two ring homomorphisms $\eta_l : R\to M$ and $\eta_r : R\to M$ which are central (i.e. two different $R$-algebra structures on $M$). 
In particular it contains the category of $R$-algebras $M$ over a commutative ring $R$ as a full subcategory.
\end{lemma}

\begin{proof}
A monoid structure on a bimodule $(R;M)$ is a morphism
\[
\mu=(\mu_R;\mu_M) : (R\otimes R;M\otimes M)\longrightarrow (R;M),
\]
and a unit map $\eta=(\eta_R;\eta_M) : (\Z;\Z)\to (R;M)$, subject to the associativity and unitality axioms.
The map $\mu_R$ and the unit $\eta_R$ then endow the ring $R$ with the structure of a monoid with respect the tensor product of rings. By the Eckmann-Hilton theorem $\mu_R$ is the multiplication of $R$ and $R$ must be a commutative ring. The map $\mu_M$ is a map of $R\otimes R$-bimodules
\[
\mu_{M} : M\otimes M\longrightarrow \mu^{*}_RM,
\]
which endows $M$ with a ring structure $m\star n:=\mu_M(m\otimes n)$. The bimodule structure determines and is determined by the ring homomorphisms $\eta_l(a)=a\cdot 1$ and $\eta_r(b)=1\cdot b$ so that we have $a \cdot m = \eta_l(a) \star m$ and $m \cdot b = m \star \eta_r(b)$.
Since  $\mu_M$ is a map of left $R\otimes R$-modules we also have 
\[
\eta_l(a) \star m=a \cdot m = (1 \cdot a)\cdot (m \star 1) = (1\cdot m) \star (a\cdot 1) = m \star \eta_l(a) 
\]
which shows that $\eta_l$ is central. Similarly we see that $\eta_r$ is central. 

Conversely for arbitrary central ring morphisms 
\[
\eta_l, \eta_r: R \to M
\]
we equip $M$ with the bimodule structure $rms := \eta_l(r) \star m \star \eta_r(s)$ and one directly checks that then $\star$ is a map in $\bimod$.
\end{proof}

The monoidal structure on $\bimod$ is in fact symmetric monoidal, where the symmetry isomorphism
\[
(R\otimes_{\Z}R';M\otimes_{\Z}M')\cong (R'\otimes_{\Z}R;M'\otimes_{\Z}M)
\]
is defined by switching the factors componentwise. We immediately get
\begin{lemma}
The category of commutative monoids in $\bimod$ is isomorphic to the category of pairs $(R;M)$ where $R$ is a commutative ring and $M$ is a commutative  $R$-algebra in two different ways.
\end{lemma}

Note that in general a monoid $(R;M)$ is not an algebra over $(R;R)$. For this to happen we need the two $R$-algebra structures on $M$ to agree.

\subsection{Definition of big Witt vectors with coefficients}

In this section we give the definition of big Witt vectors with coefficients $W(R;M)$ for a (not necessarily commutative) ring $R$ and an $R$-bimodule $M$, see Definition \ref{def_Witt} below. This construction will determine a functor from the category of bimodules to the category of abelian groups.

\begin{defn}\label{defn_tensoralgebra}
For a ring $R$ and a bimodule $M$, we define the completed tensor algebra
\[
\widehat{T}(R;M) = \prod_{n\geq 0} M^{\otimes_R n}.
\]
We think of elements as representing formal power series of the form
\[
a_0 + a_1 \cdot t + a_2 \cdot t^2 + a_3\cdot t^3 + \ldots,
\]
where $a_n\in  M^{\otimes_R n}$. Note that the powers of $t$ are just notation indicating the grading, there is no element $t$.
The ring structure is defined in the obvious way, and is continuous with respect to the product topology. We also define the topological subgroup of special units $\widehat{S}(R;M)$ 
to be the multiplicative subgroup of elements with constant term $a_0=1$.
\end{defn}

The topology on the special units is explicitly given by filtering by degree. More precisely, we say that a special unit is in filtration $\geq n$ if it is of the form
\[
1 + a_n t^n + a_{n+1} t^{n+1} + \ldots
\]
We denote the subgroup of filtration $\geq n$ special units by $\widehat S^{(n)}(R;M)$. Those form a neighbourhood basis of $1$.
Observe that $\widehat S^{(n)}(R;M)/ \widehat S^{(n+1)}(R;M)$ is isomorphic to $M^{\otimes_R n}$, since modulo higher filtration, multiplication of special units of filtration $\geq n$ is just addition of the tensor length $n$ part.

Also observe that in the case $M=R$ the tensor algebra $\widehat{T}(R;M)$ is the power series ring $R\pow{t}$, and the special units are just the elements of $R\pow{t}$ with constant term $1$.

\begin{defn}\label{def_Witt}
We define a ``Teichm\"uller'' map of sets $\tau: M \to \widehat{S}(R;M)$ by sending $m\mapsto 1-mt$. We then define the abelian group of big Witt vectors as
\[
W(R;M) = \frac{ \widehat{S}(R;M)^{\mathrm{ab}}}{ \tau(rm) \sim \tau(mr)} ,
\]
where the relation runs over all possible $m\in M$ and $r\in R$, and we take the abelianisation and the quotient in Hausdorff topological groups, i.e. divide by the closure of the normal subgroup generated by the relations we impose.
\end{defn}

\begin{rem}
Throughout the paper, we treat $W(R;M)$ as a complete Hausdorff topological abelian group, see Proposition \ref{prop:coeq} and the constructions in \S\ref{sec:operators}.
Alternatively one can consistently treat $W(R;M)$ as a pro-object, or even just an inverse system, of the truncated Witt vectors discussed in detail in \S\ref{sec:truncated}.
As discussed there, all the structure maps on $W(-;-)$ we discuss are compatible with truncation in the appropriate sense, and thus can be recovered in the untruncated setting from their truncated counterparts. The approach with pro-objects is the one usually adopted when dealing with the de Rham--Witt complex.
\end{rem}

\begin{rem}
\label{rem:solid} When $R=M$ is commutative, we have that $W(R;R)$ is the multiplicative subgroup of power series with constant term one, which is the usual additive abelian group of Witt vectors $W(R)$.

Suppose more generally that $R$ is commutative and that $M$ is a solid commutative $R$-algebra, i.e. that the multiplication map $\mu : M\otimes_R M\to M$ is an isomorphism. In this case the map of bimodules $(R;M)\to (M;M)$ induced by the algebra structure gives an isomorphism of abelian groups $W(R;M)\cong W(M;M)=W(M)$ with the usual Witt vectors of $M$ as follows immediately from the definitions. For example $W(\Z;\F_p)\cong W(\F_p)$.
\end{rem}

\begin{rem}
\label{rem:additive}
The (generally noncommutative) group $\widehat{S}(R;M)$ is written multiplicatively. However, we will write the group structure on the abelian groups $W(R;M)$ additively. This should not lead to confusion, since we will use the multiplicative notation precisely if we think about elements of $W(R;M)$ as representative power series in $\widehat{S}(R;M)$.
\end{rem}

\begin{lemma}
\label{lem:generators}
$\widehat{S}(R;M)$ is topologically generated by elements of the form $(1 + x_0 \otimes\cdots\otimes x_{k-1} t^k)$. More generally, given a generating set $G_k \subseteq M^{\otimes_R k}$ (as abelian groups) for every $k$, the group $\widehat{S}(R;M)$ is topologically generated by elements of the form $(1+ g_k t^k)$ with $g_k\in G_k$. 
\end{lemma}
\begin{proof}
Assume we have a special unit in filtration $\geq n$, i.e. one of the form
\[
1 + a_n t^n + a_{n+1} t^{n+1} + \ldots
\]
Then the coefficient $a_n$ can be written as a finite sum of elements in $G_n$, and we can split off corresponding factors of the form $(1 + g_n t^n)$. This allows us to write any such special unit as a product of ones of the form $(1 + g_n t^n)$ and a remainder term of higher filtration. Inductively, this proves that, up to a remainder term of arbitrarily high filtration, any element of $\widehat{S}(R;M)$ can be written as a product of terms of the form $(1 + g_k t^k)$. This proves the claim.
\end{proof}

\begin{lemma}
\label{lem:complete}
The filtration of $\widehat{S}(R;M)$ by the $\widehat{S}^{(n)}(R;M)$ induces a filtration $W^{(n)}(R;M)$ on the quotient $W(R;M)$. This filtration is complete and Hausdorff. 
\end{lemma}
\begin{proof}
Observe that the kernel $N$ of $\widehat{S}(R;M)\to W(R;M)$ is by definition closed, so its filtration by the $N\cap \widehat{S}^{(n)}(R;M)$ is complete and Hausdorff, or equivalently, the derived limit $\operatorname{Rlim}_n N\cap \widehat{S}^{(n)}(R;M)$ vanishes. Since the original filtration is complete and Hausdorff, i.e. $\operatorname{Rlim}_n \widehat{S}^{(n)}(R;M) = 0$, we see that
\[
\operatorname{Rlim}_n \widehat{S}^{(n)}(R;M)/(N\cap \widehat{S}^{(n)}(R;M)) = 0,
\]
i.e. that the image filtration is complete and Hausdorff.
\end{proof}

\begin{prop}
\label{prop:coeq}
  As a functor from $\bimod$ to the category of \emph{Hausdorff} topological groups, $\widehat{S}(-;-)$ and $W(-;-)$ commute with reflexive coequalizers.
\end{prop}
\begin{proof}
  We first check that $\widehat{S}(-;-)$ commutes with reflexive coequalizers.
  To see this, we need to check that if 
  \[
    \begin{tikzcd}
      (R_1,M_1) \rar[shift left,"f"]\rar[shift right,"g"'] & (R_0,M_0) \lar\rar & (R,M)
    \end{tikzcd}
  \]
  is a reflexive coequalizer of bimodules,
  then $\widehat{S}(R;M)$ is obtained from $\widehat{S}(R_0;M_0)$ by quotienting by the closed normal subgroup $N$ generated by all $f(y)g(y)^{-1}$ for $y\in \widehat{S}(R_1;M_1)$.
  Surjectivity is clear, so we have to check that the kernel of $\widehat{S}(R_0;M_0)\to\widehat{S}(R;M)$ agrees with $N$.
  The subgroup $N$ is clearly contained in the kernel. Given an element $x$ in the kernel, it is of the form $(1+a_n t^n + \ldots)$, with $a_n$ in the kernel of the right map in the diagram
  \[
     \begin{tikzcd}
       M_1^{\otimes_{R_1} n} \rar[shift left,"f"]\rar[shift right,"g"'] &  M_0^{\otimes_{R_0} n} \lar\rar & M^{\otimes_{R} n} 
    \end{tikzcd}
  \]
  Since reflexive coequalizers of abelian groups commute with tensor products, this diagram is also a reflexive coequalizer of abelian groups, so $a_n$ is of the form $f(b_n)-g(b_n)$. Thus the original $x$ can up to a term of higher filtration (which is also in the kernel) be written as $x = f(1+b_n t^n) g(1+b_nt^n)^{-1}$. Inductively, we can write any element in the kernel as a convergent product of elements of the form $f(y)g(y)^{-1}$, so the kernel is contained in $N$ as desired.
  We now want to show that $W(-;-)$ also commutes with reflexive coequalizers.
  To that end, let $N(R;M)$ denote the closed normal subgroup of $\widehat{S}(R;M)$ generated by commutators and elements of the form $(1+rmt)(1+mrt)^{-1}$, so that $W(R;M) = \widehat{S}(R;M)/N(R;M)$.
  Since $\widehat{S}(-)$ commutes with reflexive coequalizers, we see that the coequalizer of $W(R_1;M_1)$ and $W(R_0;M_0)$ can be described as the quotient of $S(R;M)$ by the closure of the image of $N(R_0;M_0)$. So we have to check that this closure agrees with $N(R;M)$. But this is clear: $N(R;M)$ is topologically generated by commutators and elements of the form $(1+rmt)(1+mrt)^{-1}$, all of which are in the image.
\end{proof}

We want to define a version with coefficients of the ghost map of the usual Witt vectors. We start by defining a map $\log: \widehat{S}(R;M)\to \Q\widehat{\otimes}\,\widehat{T}(R;M)$, where $\widehat{\otimes}$ denotes the completed tensor product $\Q\widehat{\otimes} \prod_{n\geq 0} M^{\otimes_R n} = \prod_{n\geq 0} \Q\otimes M^{\otimes_R n}$, by
\[
\log(1 + f) = f - \frac{1}{2} f^2 + \frac{1}{3} f^3 - \ldots
\]
We will also use $\log$ to refer to the map $\widehat{S}(R;M)\to \Q\widehat{\otimes} \prod_{n\geq 1} M^{\circledcirc_R n}$ obtained by postcomposing with the quotient map $\widehat{T}(R;M) \to \prod_{n\geq 1} M^{\circledcirc_R n}$ to the cyclic tensor product of \S\ref{secprelim}.

A basic observation from algebra is that the derivative $\frac{d}{dt} \log (1+f(t))$ over a commutative ring has integral coefficients, because it agrees with $f'\cdot (1+ f)^{-1}$. The key property of derivation is that the coefficient in front of $x^n$ is multiplied by $n$.
In our setting with coefficients, it turns out that the correct analogue of multiplication with $n$ is the transfer (i.e. additive norm) with respect to the $C_n$ action on the abelian group $M^{\circledcirc_R n}$.

Define $\tr: \prod_{n\geq 1} (M^{\circledcirc_R n})_{C_n} \to \prod_{n\geq 1} (M^{\circledcirc_R n})^{C_n}$ to be the product of the transfers of the $C_n$ action on $M^{\circledcirc_R n}$. We define a map 
\[
\tlog=-\tr\circ \log: \widehat{S}(R;M)\longrightarrow \Q\widehat{\otimes} \prod_{n\geq 1} (M^{\circledcirc_R n})^{C_n}.
\]
Note that for $R$ a commutative ring and $M=R$, $\tlog$ agrees with $-t\cdot \operatorname{dlog}$, the operator that sends a power series $1+f(t)$ to $-t$ times the derivative of $\log(1+f(t))$.\footnote{Note that the minus sign in front of $\tr\circ \log$ is a convention. There are in fact four different possible conventions that one can adopt here, which lead to slightly different formulas in what follows. Also see Remark 1.15 in \cite{MR3316757} for a  discussion.}

\begin{prop}\label{prop:tlogadd}\label{prop:tlogspecialvalues}
The map $\tlog:\widehat{S}(R;M)\to \Q\widehat{\otimes} \prod_{n\geq 1} (M^{\circledcirc_R n})^{C_n}$ satisfies the following properties:
\begin{enumerate}
\item It is a homomorphism with respect to the group structures given by multiplication in the domain, and addition in the codomain,
\item It sends $1 - a_n t^n$ to the element
\[
\tlog(1 - a_n t^n)=\tr^{C_n}_e a_n t^n + \tr^{C_{2n}}_{C_2} a_n^2 t^{2n} + \tr^{C_{3n}}_{C_3}a_n^3 t^{3n} + \ldots,
\]
and in particular for $n=1$ we get that $\tlog(1 - a_1 t)=  a_1 t + a_1^2 t^{2} + a_1^3 t^{3} + \ldots$.
\item It satisfies
\[
\tlog(1 - fg) = \tlog(1-gf)
\]
for any elements $f,g\in \widehat{T}(R;M)$, at least one of which has trivial constant term.
\item It lifts uniquely along the rationalisation map to a natural homomorphism
\[
\tlog: \quad \widehat{S}(R;M) \longrightarrow  \prod_{n\geq 1} (M^{\circledcirc_R n})^{C_n},
\]
which still has the above properties. Here naturality is with respect to the category of bimodules $(R;M)$.
\end{enumerate}
\end{prop}
\begin{proof}
For the first claim it suffices to show that $\log: \widehat{S}(R;M) \to \Q\widehat{\otimes} \prod_{n\geq 1} (M^{\circledcirc_R n})_{C_n}$ is a homomorphism.
We define an operator $\partial: \widehat{T}(R;M)\to \widehat{T}(R;M)$ that acts by multiplication with $n$ on the factor $M^{\otimes_R n}$. This satisfies $\partial(fg) = \partial f\cdot g + f\cdot \partial g$. In particular, we have
\[
\partial f^n = (\partial f) f^{n-1} + f (\partial f) f^{n-2} + \ldots + f^{n-1} (\partial f).
\]
Now let us write $f \sim g$ when elements $f, g\in \Q\widehat{\otimes}\, \widehat{T}(R;M)$ have the same image under the canonical map to $\Q\widehat{\otimes} \prod_{n\geq 1} (M^{\circledcirc_R n})_{C_n}$. One easily sees by expanding that $fg \sim gf$ for any elements $f,g$. It follows that $\partial f^n \sim n (\partial f) f^{n-1}$, and for a special unit $u=1+f$:
\begin{align*}
\partial \log u &\sim (\partial f) - (\partial f)f + (\partial f)f^2 - \ldots
                 = (\partial f) \cdot (1 + f)^{-1}
                 = (\partial u) \cdot u^{-1}.
\end{align*}
Therefore, for any special units $u,v$, we see that
\begin{align*}
\partial \log(uv) &\sim \partial(uv) \cdot (uv)^{-1}
 = ((\partial u) \cdot v + u\cdot (\partial v)) v^{-1} u^{-1}
 = (\partial u) \cdot u^{-1} + u \cdot (\partial v) \cdot v^{-1} \cdot u^{-1}\\
 &= \partial \log u + u \cdot (\partial \log v) \cdot u^{-1}\sim \partial \log u +  \partial \log v.
\end{align*}
This shows that in the diagram
\[
\begin{tikzcd}
\widehat{S}(R;M)\rar{\log} & \Q\widehat{\otimes}\widehat{T}(R;M) \rar\dar{\partial} & \Q\widehat{\otimes}\prod_{n\geq 1} (M^{\circledcirc_R n})_{C_n}\dar{\partial}\\
  & \Q\widehat{\otimes}\widehat{T}(R;M) \rar & \Q\widehat{\otimes}\prod_{n\geq 1}(M^{\circledcirc_R n})_{C_n}
\end{tikzcd}
\]
the lower composite from the left most node to the lower right node is a homomorphism. But since the rightmost vertical map is an isomorphism, the top horizontal composite is a homomorphism as well.

For the second claim we calculate explicitly
\begin{align*}
\tlog(1-a_nt^n) &= \tr_{e}^{C_{n}} a_n t^n + \tr_{e}^{C_{2n}} \frac{a_n^2}{2} t^{2n}  + \tr^{C_{3n}}_{e}\frac{a_n^3}{3} t^{3n} + \ldots\\
&= \tr_{e}^{C_{n}} a_n t^n + \tr_{C_2}^{C_{2n}} \tr_{e}^{C_{2}} \frac{a_n^2}{2} t^{2n} + \tr^{C_{3n}}_{C_3} \tr^{C_{3}}_{e}\frac{a_n^3}{3} t^{3n} + \ldots\\
&= \tr^{C_n}_e a_n t^n + \tr^{C_{2n}}_{C_2} a_n^2 t^{2n} + \tr^{C_{3n}}_{C_3}a_n^3 t^{3n} + \ldots
\end{align*}
where the last equality comes from the fact that $a_n^k$ is already invariant under the action of the subgroup $C_k \subseteq C_{nk}$, so $\tr_{e}^{C_k}$ just acts by multiplication with $k$.

For the third claim, it suffices again to check this for the map $\log$. We have
\[
\log(1-fg) = -fg - \frac{fgfg}{2} - \ldots \sim -gf - \frac{gfgf}{2} - \ldots = \log(1-gf),
\]
and thus they agree in $\Q\widehat{\otimes}\prod_{n\geq 1} (M^{\circledcirc_R n})_{C_n}$.

For the last claim, we first observe that the image of $\tlog$ is integral, i.e. contained in the image of the rationalisation $\prod_{n\geq 1} (M^{\circledcirc_R n})^{C_n}\to \Q\widehat{\otimes}\prod_{n\geq 1} (M^{\circledcirc_R n})^{C_n}$. Since $\widehat{S}(R;M)$ is topologically generated by elements of the form $(1+a_nt^n)$, this follows immediately from the first two claims.
For a pair $(R;M)$ where $(M^{\circledcirc_R n})^{C_n}$ is torsion free, the rationalisation is injective.
So on the full subcategory of those $(R;M)$ with torsion free $(M^{\circledcirc_R n})^{C_n}$, $\tlog$ factors to a unique natural transformation as desired. As we are mapping to a Hausdorff topological group, $\widehat{S}(R;M)$ commutes with reflexive coequalizers in Hausdorff topological groups, and we can resolve every bimodule $(R;M)$ as a reflexive coequalizer of $(R_1;M_1)$ and $(R_0;M_0)$ with torsion-free $({M_i}^{\circledcirc_{R_i} n})^{C_n}$ (see Lemma \ref{restorfree}), this natural transformation extends uniquely to all $(R;M)$.
\end{proof}

We now want to show that $\tlog$ descends to the Witt vectors $W(R;M)$.

\begin{lemma}\label{lem:tlogkernel}
Suppose $(R;M)$ is a bimodule with the property that the transfer maps $\tr: (M^{\circledcirc_R n})_{C_n} \to (M^{\circledcirc_R n})^{C_n}$ are injective for all $n$. Suppose further that $G\subseteq \widehat{S}(R;M)$ is a subgroup with the following properties:
\begin{enumerate}
\item $G$ is closed.
\item $G$ is contained in the kernel of $\tlog:\widehat{S}(R;M) \to \prod_{n\geq 1} M^{\circledcirc_R n}$.
\item For each $n$, each $i,j\geq 0$ with $i+j=n$, and each $x_i \in M^{\otimes_R i}$, $y_j\in M^{\otimes_R j}$, $G$ contains an element of the form $(1 - (x_i \otimes y_j - y_j\otimes x_i) t^{i+j} + \ldots)$.
\end{enumerate}
Then $G$ agrees with the kernel of $\tlog: \widehat{S}(R;M) \to \Q\widehat{\otimes}\prod_{n\geq 1} M^{\circledcirc_R n}$.
\end{lemma}
\begin{proof}
We have to show that every element in the kernel of $\tlog$ can be written as a convergent product of elements in $G$. Suppose we have an element of the form $f_n=(1+a_nt^n + \ldots)$ in the kernel of $\tlog$, with $a_n\in M^{\otimes_R n}$. Then, since 
\[
\tlog(1 + a_n t^n + \ldots) = - \tr_e^{C_n} a_n t^n + \ldots,
\]
we have that $a_n$ is in the kernel of the composite $M^{\otimes_R n} \to (M^{\circledcirc_R n})_{C_n} \to (M^{\circledcirc_R n})^{C_n}$. Since we assumed the latter map to be injective, $a_n$ is in the kernel of the quotient map $M^{\otimes_R n}\to (M^{\circledcirc_R n})_{C_n}$. This kernel is generated by differences of the form $x_i \otimes y_j - y_j \otimes x_i$ for $i+j=n$, with $x_i \in M^{\otimes_R i}$ and $y_j\in M^{\otimes_R j}$, so $a_n$ can be written as a sum of such elements. Now by (3), this implies that we can write $f_n$ as a product of elements in $G$ of filtration $\geq n$, and a remainder term of filtration $\geq n+1$, which by (2) is also in the kernel of $\tlog$. Iterating this argument, (1) implies that every element in the kernel of $\tlog$ is in $G$.
\end{proof}

\begin{lemma}
\label{lem:commutator}
We have the following description for the leading term of a commutator:
\[
[(1+a_n t^n+\ldots), (1+b_m t^m+\ldots)] = 1 + (a_nb_m - b_ma_n)\cdot t^{n+m} + \ldots
\]
\end{lemma}
\begin{proof}
We first compute the leading term for a commutator of $(1+a_n t^n)$ and $(1+b_m t^m)$. We have
\begin{align*}
(1+a_n t^n)(1+b_m t^m) &= (1 + a_n t^n + b_m t^m + a_n b_m t^{n+m})\\
 &= (1 + a_n b_m t^{n+m} + \ldots) (1 + a_n t^n + b_m t^m).
\end{align*}
Multiplying this with the inverse of $(1+b_m t^m)(1+a_n t^n)$, we obtain
\begin{align*}
[(1+a_n t^n), (1+b_m t^m)] &= (1 + a_n b_m t^{n+m} + \ldots) \cdot (1 +  b_ma_n t^{n+m} + \ldots)^{-1}\\
 &= 1 + (a_nb_m - b_ma_n)\cdot t^{n+m}+\ldots.
\end{align*}
In particular, this shows that elements $(1+a_k t^k)$ and $(1+b_l t^l)$ commute up to terms of filtration $\geq k+l$. By continuity, we also get that arbitrary elements of filtration $\geq k$ and $\geq l$ commute up to terms of filtration $\geq k+l$. So if we have
\begin{gather*}
x = (1 + a_n t^n + \ldots) = (1+ a_n t^n) \cdot x',\\
y = (1+ b_m t^m + \ldots) = (1+b_m t^m) \cdot y',
\end{gather*}
with $x'$ of filtration $>n$, and $y'$ of filtration $>m$, we see that, up to terms of filtration $>n+m$, $x'$ commutes with $(1+b_mt^m)$, $y'$ commutes with $(1+a_nt^n)$, and $x'$ commutes with $y'$. We thus get that $[x,y]$ and $[(1+a_n t^n), (1+b_mt^m)]$ agree up to order $n+m$, from which the result follows.
\end{proof}

This may suggest that the associated graded of the filtration $W^{(n)}(R;M)$ is given by $(M^{\circledcirc_R n})_{C_n}$ in degree $n$. However, it can be smaller than that. An example with $R=M$ can be found in \cite{HesselholtncWcorr}.

\begin{prop}
\label{prop:ghost} The map $\tlog$ descends to a continuous group homomorphism 
\[\tlog : W(R;M) \longrightarrow \prod_{n\geq 1} (M^{\circledcirc_R n})^{C_n},\] which we call the ghost map. If all the transfer maps $(M^{\circledcirc_R n})_{C_n}\to (M^{\circledcirc_R n})^{C_n}$ are injective (for example if the $(M^{\circledcirc_R n})_{C_n}$ are torsion free), this map is injective, and in fact a homeomorphism onto its image.
\end{prop}
\begin{proof}
The map clearly factors through the abelianisation, and by Proposition \ref{prop:tlogspecialvalues} (3), we have $\tlog(1- rm \cdot t) = \tlog(1-mr\cdot t)$, so it factors through $W(R;M)$.

For injectivity, note that by Lemmas \ref{lem:tlogkernel} and \ref{lem:commutator}, the closed subgroup generated by commutators and elements of the form $(1-rm\cdot t)(1-mr\cdot t)^{-1}$ actually agrees with the kernel of $\tlog$ if the transfers are injective.

For the last part, it suffices to check the following stronger version of injectivity: If an element $f\in W(R;M)$ has the property that $\tlog f\in \prod_{n\geq 1} (M^{\circledcirc_R n})^{C_n}$ has filtration at least $k$, then $f$ has filtration at least $k$ as well. But observe that this is exactly what the argument in the proof of \ref{lem:tlogkernel} gives us.
\end{proof}

\begin{lemma}
\label{lem:otherrelations}
$W(R;M)$ agrees with the quotient of $\widehat{S}(R;M)$ by any of the following:
\begin{enumerate}
\item The closed subgroup generated by commutators and all elements of the form $(1+rmt) (1+mrt)^{-1}$ for $r\in R$ and $m\in M$. (These are the relations that appear in our definition of $W(R;M)$, we recall them here for convenience.)
\item The closed subgroup generated by commutators and all elements of the form
\[
(1 - rf) \cdot (1 - fr)^{-1}
\]
where $r\in R$ and $f\in \widehat{T}(R;M)$ with trivial constant term.
\item The closed subgroup generated by all elements of the form
\[
(1 - x_i y_j t^{i+j}) \cdot (1 - y_j x_i t^{i+j})^{-1},
\]
with $i+j\geq 1$. We allow $i=0$ or $j=0$, e.g. $x_0\in R$, so this relation includes the Teichm\"uller relations $\tau(rx)\tau(xr)^{-1}$.
\item The closed subgroup generated by all elements of the form
\[
(1 - fg) \cdot (1 - gf)^{-1}
\]
for elements $f,g\in \widehat{T}(R;M)$ with $f$ or $g$ having trivial constant term.
\end{enumerate}
\end{lemma}
\begin{proof}
We have to show that all these subgroups of $\widehat{S}(R;M)$ agree. By resolving via reflexive coequalizers, we can reduce to the case where the transfers $\tr: (M^{\circledcirc_R n})_{C_n} \to (M^{\circledcirc_R n})^{C_n}$ are injective. In that case, we claim they all agree with the kernel of $\tlog: \widehat{S}(R;M) \to \prod_{n\geq 1} M^{\circledcirc_R n}$. By Proposition \ref{prop:tlogspecialvalues}, they are all contained in the kernel of $\tlog$, and using Lemma \ref{lem:commutator}, we see that they also satisfy condition (3) of Lemma \ref{lem:tlogkernel}, which then implies the claim.
\end{proof}

\begin{lemma}
\label{lem:products}
$W(-;-)$ commutes with finite products, i.e. given pairs $(R; M)$ and $(S; N)$, the canonical map
\[
W(R\times S; M\times N) \to W(R;M)\times W(S;N)
\]
is an isomorphism.
\end{lemma}
\begin{proof}
The map $\widehat{S}(R\times S;M\times N) \to \widehat{S}(R;M)\times \widehat{S}(S;N)$ is an isomorphism of topological groups. By Lemma \ref{lem:otherrelations}, it suffices to check that it sends the closed subgroups generated by elements of the form $(1 - x_i y_j t^{i+j}) \cdot (1 - y_j x_i t^{i+j})^{-1}$ to each other. This follows from
\begin{align*}
&(1 - (a_i,b_i) (a_j,b_j) t^{i+j}) \cdot (1 - (a_j,b_j) (a_i,b_i) t^{i+j})^{-1}\\
=\ &(1- (a_i,0) (a_j,0) t^{i+j}) \cdot (1 - (a_j,0)(a_i,0) t^{i+j})^{-1}\\
&\cdot (1- (0,b_i) (0,b_j) t^{i+j}) \cdot (1 - (0,b_j)(0,b_i) t^{i+j})^{-1}. \qedhere
\end{align*}
\end{proof}

\subsection{The operators and the monoidal structure}\label{sec:operators}

We now construct additional structure on the big Witt vectors with coefficients: Verschiebung maps 
\[
V_n : W(R; M^{\otimes_R n}) \to W(R; M),
\]
Frobenius maps 
\[
F_n : W(R; M)\to W(R; M^{\otimes_R n}),
\]
a $C_n$-action on $W(R; M^{\otimes_R n})$, and a lax symmetric monoidal structure, i.e. external products
\[
\star: W(R;M) \otimes W(S;N) \to W(R \otimes S; M \otimes N) \ .
\]
To do so, we first discuss a preferred set of generators of $W(R;M)$.

\begin{defn}
We let $\tau_n: M^{\times n} \to W(R;M)$ be the map
\[
\tau_n(m_1,\ldots, m_n) = (1 - m_1 \otimes \cdots \otimes m_n t^n).
\]
\end{defn}

\begin{lemma}
\label{lem:teichmullerinvariance}
The images of the $\tau_n$ generate $W(R;M)$ topologically. The maps $\tau_n$ are cyclically invariant, meaning that
\[
\tau_n(m_1,\ldots, m_n) = \tau_n(m_{\sigma(1)}, \ldots, m_{\sigma(n)})
\]
for any $\sigma\in C_n$, and they satisfy
\begin{align*}
\tau_n(m_1,\ldots, m_i r, m_{i+1}, \ldots,m_n) &= \tau_n(m_1,\ldots, m_i, rm_{i+1}, \ldots,m_n),\\
\tau_n(rm_1, \ldots, m_n) &= \tau_n(m_1, \ldots, m_n r).
\end{align*}
\end{lemma}
\begin{proof}
This follows immediately from Lemma \ref{lem:generators} and Lemma \ref{lem:otherrelations}.
\end{proof}
In spite of the identities of Lemma $\ref{lem:teichmullerinvariance}$, $\tau_n$ does not descend to the cyclic tensor power since it is not additive. It is however well-defined on the tensor power $M^{\otimes_R n}$, by definition. We will sometimes abuse notation and apply $\tau_n$ to an element of $M^{\otimes_R n}$.

\begin{prop}
\label{prop:verschiebung}
There are continuous Verschiebung homomorphisms
\[
V_n: W(R; M^{\otimes_R n}) \to W(R; M)
\]
for every $n\geq 1$, uniquely characterized by the commutativity of the diagrams
\[
\begin{tikzcd}
 M^{\times n k} \rar\dar{\id} & (M^{\otimes_R n})^{\times k} \rar{\tau_k} & W(R; M^{\otimes_R n})\dar{V_n}\\
 M^{\times n k} \ar[rr,"\tau_{nk}"] & & W(R; M)\rlap{\ .}
\end{tikzcd}
\]
Under the ghost map, $V_n$ is compatible with the additive map
\[
\prod_{k\geq 1} (M^{\circledcirc_R nk})^{C_k} \longrightarrow \prod_{k\geq 1} (M^{\circledcirc_R k})^{C_k} 
\]
given on the factor $(M^{\circledcirc_R nk})^{C_k}$ by the transfer $\tr_{C_k}^{C_{nk}}$ to $(M^{\circledcirc_R nk})^{C_{kn}}$.
They satisfy $V_n V_m = V_{nm}$ as maps $W(R; M^{\otimes_R nm}) \to W(R; M)$.
\end{prop}
\begin{proof}
Since the images of the maps $M^{\times n k}\to W(R; M^{\otimes_R n})$ topologically generate $W(R; M^{\otimes_R n})$, there is at most one $V_n$ with the desired properties.

For the existence, consider that the homomorphism
$
\widehat{S}(R;M^{\otimes_R n}) \to \widehat{S}(R;M)
$
given by sending
\[
1 + \sum_i a_i t^i \mapsto 1 + \sum_i a_i t^{ni}
\]
preserves the relations given in Lemma \ref{lem:otherrelations}, which were of the form
\[
(1 - x_i y_j t^{i+j}) \sim (1 - y_j x_i t^{i+j}).
\]
Thus, this homomorphism factors to a homomorphism $V_n: W(R; M^{\otimes_R n}) \to W(R; M)$ as desired. Next, we compute that this $V_n$ is compatible with the given description on ghosts. But it suffices to check this on generators. The ghost map sends
\begin{gather*}
\tlog(1 - a_k t^k) = \tr_e^{C_k} a_k t^k + \tr_{C_2}^{C_{2k}} a_k^2 t^{2k} + \ldots,\\
\tlog (V_n(1 - a_k t^k)) = \tlog(1 - a_k t^{nk}) = \tr_e^{C_{nk}} a_k t^{nk} + \tr_{C_2}^{C_{2nk}} a_k^2 t^{2nk} + \ldots.
\end{gather*}
As $\tr_{C_{ik}}^{C_{nik}} \tr_{C_i}^{C_{ik}} = \tr_{C_i}^{C_{nik}}$, the described map on ghosts sends $\tlog(1 - a_k t^k)$ to $\tlog(V_n(1-a_k t^k))$.

Finally, to check that $V_n V_m = V_{nm}$, it suffices that they agree on the image of the $\tau_k$, which follows from the defining properties of the $V_i$.
\end{proof}
Note that this implies in particular that $\tau_k: M^{\times k} \to W(R;M)$ agrees with the composite
\[
\tau_k: M^{\times k} \to M^{\otimes_R k} \xrightarrow{\tau} W(R; M^{\otimes_R k}) \xrightarrow{V_k} W(R;M).
\]

\begin{prop}
\label{prop:weilaction}
There is a continuous homomorphism $\sigma: W(R; M^{\otimes_R n}) \to W(R;M^{\otimes_R n})$, uniquely characterized by the commutativity of the diagrams
\[
\begin{tikzcd}
 M^{\times n k} \rar{}\dar{\mathrm{\sigma}} & (M^{\otimes_R n})^{\times k} \rar{\tau_k} & W(R; M^{\otimes_R n})\dar{\sigma}\\
 M^{\times n k} \rar{}& (M^{\otimes_R n})^{\times k}\rar{\tau_k} & W(R; M^{\otimes_R n}),
\end{tikzcd}
\]
  where the left vertical map is given by $(m_1,\ldots, m_{nk-1}, m_{nk})\mapsto (m_{nk},m_1,\ldots,m_{nk-1})$. It has order $n$, and thus gives a $C_n$-action on $W(R; M^{\otimes n})$, which we refer to as Weyl action. This is compatible with the $C_n$-action on ghost components $\prod_{k\geq 1} (M^{\circledcirc nk})^{C_k}$ obtained degreewise as the residual action of $C_n \cong C_{nk}/C_k$.
\end{prop}
\begin{proof}
Again, the images of the upper horizontal maps (jointly for all $k$) generate $W(R; M^{\otimes_R n})$ topologically, and so there is at most one homomorphism $\sigma$.
  To see one exists, it is sufficient to do so for $(R;M)$ with torsion-free $(M^{\circledcirc nk})_{C_k}$, since the target is Hausdorff and we can resolve any $(R;M)$ as a reflexive coequalizer of $(R_0;M_0)$ and $(R_1; M_1)$ with torsion-free $(M_i^{\circledcirc nk})_{C_k}$.

In the torsion-free case, we know by Proposition \ref{prop:ghost} that $\tlog$ is a homeomorphism onto its image.
It is therefore sufficient to check that the described $C_n$-action on ghost components restricts to an action on the image of $\tlog$, or more precisely sends $\tlog(1 - m_1 \otimes \ldots \otimes m_{nk} t^k)$ to $\tlog(1 - m_{nk} \otimes m_1\otimes \ldots \otimes m_{nk-1}t^k)$. 

The $ik$-th coefficient of $\tlog(1 - m_1 \otimes \ldots \otimes m_{nk} t^k)$ is given (Prop. \ref{prop:tlogspecialvalues}) by 
\[
\tr_{C_i}^{C_{ik}} (m_1 \otimes \ldots \otimes m_{nk})^{\otimes i}, 
\]
which is shifted by a generator of $C_{nik}$ (representing the residual action of a generator of $C_n = C_{nik}/C_{ik}$) to the element
\[
\tr_{C_i}^{C_{ik}} (m_{nk} \otimes m_1\otimes \ldots \otimes m_{nk-1})^{\otimes i},
\]
which is the $ik$-th coefficient of $\tlog(1 -m_{nk} \otimes m_1\otimes \ldots \otimes m_{nk-1} t^k)$.

The $n$-th power of $\sigma$ acts as identity on ghost components, and because of naturality, this implies that $\sigma$ always has order $n$.
\end{proof}

\begin{prop}
\label{prop:frobenius}
There are continuous Frobenius homomorphisms
\[
F_n: W(R; M) \longrightarrow W(R; M^{\otimes_R n})
\]
uniquely characterized by the commutativity of the diagrams
\[
\begin{tikzcd}
M^{\times k}\dar{(-)^{\times n/d}} \ar[rr,"\tau_k"] & & W(R;M)\dar{F_n}\\
M^{\times{kn/d}} \rar{\tau_{k/d}} & W(R; M^{\otimes_R n})\rar{\sum\limits_{\sigma\in C_d} \sigma} & W(R; M^{\otimes_R n})\rlap{\ .}
\end{tikzcd}
\]
 Here $d$ is the greatest common divisor of $k$ and $n$, the left vertical map sends $(m_1, \ldots, m_k)$ to $(m_1, \ldots, m_k, \ldots, m_1, \ldots, m_k)$ (i.e. $\frac{n}{d}$ consecutive blocks of $(m_1,\ldots, m_k)$), and the sum on the lower right is over the subgroup $C_d \subseteq C_n$.

$F_n$ is compatible with the map $\prod_{k\geq 1} (M^{\circledcirc_R k})^{C_k} \to \prod_{k\geq 1} (M^{\circledcirc_R kn})^{C_k}$ that projects away factors whose index is not divisible by $n$, and includes $(M^{\circledcirc_R {kn}})^{C_{kn}}$  into $(M^{\circledcirc_R kn})^{C_k}$.
\end{prop}
\begin{proof}
As the given diagram determines $F_n$ on the images of all the $\tau_k$, which topologically generate $W(R;M)$, there is at most one such $F_n$.

Existence can again be checked in the case where $\tlog$ is an embedding. There, we first check that the described map on ghost components sends $\tlog(1 - m_1\otimes \cdots \otimes m_k  t^k)$ to the value compatible with the commutative diagram in the claim. As
\[
\tlog(1 - a_k t^k) = \sum_i \tr_{C_i}^{C_{ki}} a_k^i t^{ik},
\]
a sum whose summands are of degrees divisible by $k$, if we pick out the summands whose degree is divisible by $n$ (and put them in degrees divided by $n$), we get a sum
\begin{equation}
\label{uppertlog}
\sum_i\tr_{C_{in/d}}^{C_{ink/d}} a_k^{in/d} t^{ik/d}
\end{equation}
ranging over degrees which are multiples of the least common multiple $\frac{nk}{d}$ of $n$ and $k$ (with $d$ again the greatest common divisor). 

This needs to match with the transfers of the residual $C_n$-actions on  $\tlog(1-a_k^{n/d} t^{k/d})$. We observe that $\tlog(1-a_k^{n/d} t^{k/d})$ is given by
\begin{equation}
\label{eq:lowertlog}
\sum_i\tr_{C_{i}}^{C_{ik/d}} a_k^{in/d} t^{ik/d}.
\end{equation}
To each term we apply the transfer $\tr_{C_{ik/d}}^{C_{ik}}$  of the residual action of $C_n \cong (C_{ink/d})/(C_{ik/d})$. We obtain
\begin{equation}
  \label{eq:lowertlogtransfer}
\sum_i\tr_{C_{i}}^{C_{ik}} a_k^{in/d} t^{ik/d}.
\end{equation}
We now use that, whenever they are both defined (i.e. on $C_{in/d}$-fixed points), the transfers $\tr_{C_{in/d}}^{C_{ink/d}}$ and $\tr_{C_i}^{C_{ik}}$ agree, since in the diagram
\[
 \begin{tikzcd}
  C_i \dar\rar & C_{ik} \dar \\
  C_{in/d} \rar & C_{ink/d} 
 \end{tikzcd}
\]
  the induced map on the cokernels of the rows is an isomorphism. So \eqref{uppertlog} and \eqref{eq:lowertlogtransfer} agree.

We have just shown that for any $a_k \in M^{\times k}$, $\tlog(1 - a_k t^k)$ is sent by the homomorphism $\prod_{k\geq 1} (M^{\circledcirc_R k})^{C_k} \to \prod_{k\geq 1} (M^{\circledcirc_R kn})^{C_k}$ to $\tr_e^{C_n} \tlog(1 - a_k^{n/d} t^{k/d})$, the transfer taken for the residual action of $C_n$.
  This means that this homomorphism restricts to a map $F_n: W(R; M) \to W(R; M^{\otimes n})$ if $\tlog$ is an embedding.
This map $F_n$ furthermore satisfies the claimed commutative diagrams (as we just proved the corresponding statement on ghosts.)
\end{proof}

The final piece of structure we want to discuss regards multiplicativity. The Witt vectors of commutative ring admit a natural ring structure, which is not present in the general case of a possibly noncommutative ring and a possibly nontrivial coefficient bimodule. Rather, we will see that $W(-;-)$ is lax symmetric monoidal as a functor of bimodules. For $R$ a commutative ring, this lax symmetric monoidal structure gives rise to a commutative ring structure on $W(R)$, as the composite map
\[
W(R)\otimes W(R) \to W(R\otimes R) \stackrel{\mu_*}{\longrightarrow} W(R),
\]
since for a commutative ring the multiplication map $\mu : R\otimes R\to R$ is a ring homomorphism (see also Corollary \ref{Wring}).

We recall from \S\ref{secprelim} that the tensor product of two bimodules $(R;M)$ and $(S;N)$ is $(R\otimes S;M\otimes N)$, where the tensor products are over $\mathbb{Z}$.

\begin{prop}
\label{prop:laxmonoidal}
The functor $W(-;-) :  \bimod\to Ab$ admits a lax symmetric monoidal structure, where the maps
\[
W(R;M) \otimes W(S; N) \xrightarrow{*} W(R\otimes S; M\otimes N)
\]
correspond to continuous bilinear maps
\[
W(R;M) \times W(S; N) \xrightarrow{*} W(R\otimes S; M\otimes N)
\]
uniquely characterized by the formula
\[
\tau_k(a_k) * \tau_l(b_l) = \sum_{\sigma\in C_d} \tau_{kl/d}(s(a_k^{\times l/d} \times (\sigma b_l)^{\times k/d})),
\]
for all $a_k \in M^{\times k}$, $b_l\in N^{\times l}$, where $d$ is the greatest common divisor of $k$ and $l$, and $s$ refers to the shuffle map $M^{\times kl/d} \times N^{\times kl/d}\to (M\otimes N)^{\times kl/d}$.
The map $W(R;M) \otimes W(S, N) \to W(R\otimes S, M\otimes N)$ is compatible on ghost components with the map
\[
\prod_{n\geq 1} (M^{\circledcirc_R n})^{C_n} \otimes \prod_{n\geq 1} (N^{\circledcirc_S n})^{C_n}\to \prod_{n\geq 1} ((M\otimes N)^{\circledcirc_{R\otimes S} n})^{C_n}
\]
which is given by the shuffle 
$
(M^{\circledcirc_R n})^{C_n} \otimes (N^{\circledcirc_S n})^{C_n} \to ((M\otimes N)^{\circledcirc_{R\otimes S} n})^{C_n}
$.
\end{prop}
\begin{proof}
  Uniqueness again follows from the fact that the images of the $\tau_k$ form a set of topological generators. Since a reflexive coequalizer diagram in Hausdorff abelian groups is also an underlying reflexive coequalizer diagram in Hausdorff spaces, and reflexive coequalizers in Hausdorff abelian groups commute with finite products, if we choose resolutions of $(R;M)$ and $(S;N)$ by reflexive coequalizers, the diagram
  \[
    \begin{tikzcd}
      W(R_1;M_1)\times W(S_1;N_1)\rar[shift left,"f"]\rar[shift right,"g"'] & W(R_0,M_0)\times W(S_0;N_0) \lar\rar & W(R;M) \times W(S;N)
    \end{tikzcd}
  \]
  is a reflexive coequalizer diagram in Hausdorff spaces.  
  Thus, a continuous map $*$ as desired can be extended from the case of free rings and bimodules to all (and is then easily seen to be bilinear in general). In the free case, $\tlog$ is an embedding, and existence follows once we check that the described map on ghost components acts in a compatible way on $\tlog(\tau_k(a_k)) \otimes \tlog(\tau_l(b_l))$.

We have
\[
\tlog(\tau_k(a_k)) = \sum_i \tr_{C_i}^{C_{ik}} a_k^i t^{ki},\quad \tlog(\tau_l(b_l)) = \sum_i \tr_{C_i}^{C_{il}} b_l^i t^{li},
\]
so if we form the degreewise product (via the maps $(M^{\circledcirc_R n})^{C_n} \otimes (N^{\circledcirc_S n})^{C_n} \to ((M\otimes N)^{\circledcirc_{R\otimes S} n})^{C_n}$), we obtain
\[
\sum_i \left(\tr_{C_{il/d}}^{C_{ikl/d}} a_k^{il/d}\right) \left(\tr_{C_{ik/d}}^{C_{ikl/d}} b_l^{ik/d}\right) t^{ikl/d}.
\]
We need to show that this agrees with
\[
  \tlog\left(\sum_{\sigma\in C_d} \tau_{kl/d}(s(a_k^{\times l/d} \times (\sigma b_l)^{\times k/d}))\right),
\]
which is given by an appropriate shuffle of
\begin{align*}
\sum_{\sigma\in C_d} \sum_i \tr_{C_i}^{C_{ikl/d}} (a_k^{il/d} \otimes (\sigma b_l)^{ik/d})t^{ikl/d}
=& \sum_i \tr_{C_i}^{C_{ikl/d}} \left(a_k^{il/d} \otimes \sum_{\sigma\in C_d}(\sigma b_l)^{ik/d}\right)t^{ikl/d}\\
=& \sum_i \tr_{C_{il/d}}^{C_{ikl/d}}\tr_{C_i}^{C_{il/d}} \left(a_k^{il/d} \otimes \tr_{C_{ik/d}}^{C_{ik}} b_l^{ik/d}\right)t^{ikl/d}\\
=& \sum_i \tr_{C_{il/d}}^{C_{ikl/d}} \left(a_k^{il/d} \otimes \tr_{C_i}^{C_{il/d}}\tr_{C_{ik/d}}^{C_{ik}} b_l^{ik/d}\right)t^{ikl/d}\\
=& \sum_i \tr_{C_{il/d}}^{C_{ikl/d}} \left(a_k^{il/d} \otimes \tr^{C_{ikl/d}}_{C_{ik/d}} b_l^{ik/d}\right)t^{ikl/d}\\
=& \sum_i \left(\tr_{C_{il/d}}^{C_{ikl/d}} a_k^{il/d}\right)\otimes  \left(\tr^{C_{ikl/d}}_{C_{ik/d}} b_l^{ik/d}\right)t^{ikl/d},
\end{align*}
where the third and fifth equalities use that $\tr_H^G$ is linear with respect to multiplication with $G$-invariant elements, and the fourth equality uses that $l/d$ and $k$ are coprime as follows:
\[
\begin{tikzcd}
C_i \rar\dar & C_{ik} \dar\\
C_{il/d} \rar  & C_{ikl/d} \\
\end{tikzcd}
\]
is a bicartesian diagram of abelian groups. This means that we have a double coset formula of the form
\[
\begin{tikzcd}
A^{C_i} \dar{\tr_{C_i}^{C_{il/d}}} & A^{C_{ik}} \lar\dar{\tr_{C_{ik}}^{C_{ikl/d}}}\\
A^{C_{il/d}}  & A^{C_{ikl/d}} \lar\\
\end{tikzcd}
\]
for any group $A$ with $C_{ikl/d}$-action. In our case, we use this to see that
\[
\tr_{C_i}^{C_{il/d}}\tr_{C_{ik/d}}^{C_{ik}} b_l^{ik/d} = \tr_{C_{ik}}^{C_{ikl/d}}\tr_{C_{ik/d}}^{C_{ik}} b_l^{ik/d} = \tr^{C_{ikl/d}}_{C_{ik/d}} b_l^{ik/d}.
\]

This shows that the described map on ghost components is compatible with the claimed value of $\tau_k(a_k) * \tau_l(b_l)$. It follows that there is a natural transformation
\[
W(R;M) \otimes W(S;N) \to W(R\otimes S; M\otimes N)
\]
as claimed. The associativity and symmetry conditions of a lax symmetric monoidal structure can again be reduced to the case of injective $\tlog$, where they follow from the corresponding formula on ghost components.
\end{proof}

The following are immediate consequences of the  symmetric monoidal structure. 

\begin{cor}\label{Wmod}
Let $R$ be a commutative ring and $M$ an $R$-module (considered as an $R$-bimodule). The module structure $l_M : R\otimes M\to M$ and the multiplication $\mu_R$ of $R$ define a $W(R)$-module structure
 \[
 W(R;R)\otimes W(R;M)\stackrel{*}{\longrightarrow}W(R\otimes R;R\otimes M)\xrightarrow{(\mu_R,l_M)_\star}W(R;M).
 \]
\end{cor}
\begin{proof}
One checks that the map $(\mu_R,l_M) : (R \otimes R;R \otimes M) \to (R;M)$ is a map in $\bimod$ which is straightforward\footnote{Here one really needs that $M$ is an $R$-module considered as a bimodule as opposed to a genuine bimodule.}. Then it follows that $(R;M)$ is a module in $\bimod$ over the commutative monoid $(R; R)$. Thus the claim follows since $W(-;-)$ is lax symmetric monoidal.  
\end{proof}

\begin{cor}\label{Wring}
For  every commutative ring $R$ and every $R$-algebra $R\to M$, the multiplication maps of $R$ and $M$ define a multiplication
\[
W(R;M)\otimes W(R;M)\stackrel{*}{\longrightarrow}W(R\otimes R;M\otimes M)\xrightarrow{(\mu_R,\mu_M)_\star}W(R;M)
\]
making $W(R;M)$ into a $W(R)$-algebra. It is commutative if $M$ is commutative. \qed
\end{cor}
\begin{rem}
In the last corollary we could have allowed two different $R$-algebra structures on $M$ (cf. Lemma \ref{monbimod}) to obtain a ring structure on $W(R;M)$. But in general it would then not be a $W(R)$-algebra.
\end{rem}

\begin{cor}\label{Wdual}
Let $R$ be a commutative ring and $M$ an $R$-module with dual $M^\vee:=\hom_R(M,R)$. The evaluation map $\ev : M^\vee\otimes M\to R$ defines a $W(R)$-bilinear pairing
\[
\langle-,-\rangle : W(R;M^{\vee})\otimes W(R;M)\stackrel{*}{\longrightarrow}W(R\otimes R;M^{\vee}\otimes M)\xrightarrow{(\mu_R,\ev)_\star}W(R).
\]
\end{cor}

\begin{prop}\label{prop:formulas}
The maps $V_n$, $F_n$, the $C_n$-action and the lax symmetric-monoidal structure satisfy the following properties:
\begin{enumerate}
 \item $V_nV_m = V_{nm}$
 \item $F_nF_m = F_{nm}$
 \item $V_n: W(R; M^{\otimes n}) \to W(R; M)$ is invariant under the $C_n$-action on $W(R; M^{\otimes n})$.
 \item $F_n: W(R; M)\to W(R; M^{\otimes n})$ is invariant under the $C_n$-action on $W(R; M^{\otimes n})$.
 \item $F_nV_n: W(R;M^{\otimes n}) \to W(R; M^{\otimes n})$ is the transfer $\sum_{\sigma\in C_n} \sigma$.
 \item $F_n$ is a symmetric monoidal transformation.
 \item We have $V_n(x*F_n(y)) = V_n(x)*y$ for all  $x\in W(R; M^{\otimes n})$ and $y\in W(S; N)$.
\end{enumerate}
\end{prop}
\begin{proof}
Some of these statements can be obtained immediately from the formulas characterizing those maps on elements of the form $\tau_k(a_k)$, but alternatively, we can always reduce them to the case of injective $\tlog$, where they follow from corresponding statements on ghost components (all of which reduce to coordinate-wise application of elementary properties of the transfer maps).
\end{proof}

\begin{rem}
An immediate consequence of the Frobenius reciprocity formula $V_n(F_n(y)*x) = y* V_n(x)$ is that for a module $M$ over a commutative ring $R$, the Frobenius and the Verschiebung operators are self-dual under the pairing of \ref{Wdual}, in the sense that
\[
\langle \phi,V_n(x)\rangle=V_n\langle F_n(\phi),x\rangle
\]
for all $\phi\in W(R;M^\vee)$ and $x\in W(R;M^{\otimes_R n})$, where the $V_n$ on the right is the Verschiebung of $W(R)$.
\end{rem}

\subsection{The trace property and Morita invariance}\label{sec:Morita}

We now show that $W(R;M)$ satisfies a certain trace invariance property.
The Weyl action constructed in Proposition \ref{prop:weilaction} admits a slight generalisation, where instead of considering the $n$-fold tensor power of a bimodule, we consider $n$ bimodules over possibly different rings.
Concretely, consider rings $R_i$, and $R_i$-$R_{i+1}$-bimodules $M_{i, i+1}$. Here $i$ ranges over the numbers $0\leq i \leq n-1$ modulo $n$, i.e. the last bimodule is an $R_{n-1}$-$R_0$-bimodule.
In this situation, we can form $R_l$-$R_l$-bimodules
\[
M_{l,l+1} \otimes_{R_{l+1}} M_{l+1,l+2}\otimes_{R_{l+2}} \cdots \otimes_{R_{l-1}} M_{l-1,l}.
\]

\begin{prop}[Trace property]
\label{prop:traceinvariance}
In the situation above, there is an isomorphism
\[
  T : W(R_0; M_{0,1}\otimes_{R_1} \ldots \otimes_{R_{n-1}} M_{n-1,0})\xrightarrow{\sim} W(R_{n-1}; M_{n-1,0}\otimes_{R_0} M_{0,1} \otimes_{R_1} \ldots \otimes_{R_{n-2}} M_{n-2,n-1})
\]
uniquely characterized by the commutative diagrams 
\[
\adjustbox{scale=0.85,center}{
\begin{tikzcd}[column sep=0.5cm] 
(M_{0,1} \times \ldots \times M_{n-1,0})^{\times k} \rar\dar{\mathrm{shift}} & (M_{0,1} \otimes_{R_1} \ldots \otimes_{R_{n-1}} M_{n-1,0})^{\times k} \rar{\tau_k}
& W(R_0; M_{0,1} \otimes_{R_1} \ldots \otimes_{R_{n-1}} M_{n-1,0})\dar{T}\\
(M_{n-1,0} \times \ldots \times M_{n-2,n-1})^{\times k} \rar & (M_{n-1,0} \otimes_{R_0} \ldots \otimes_{R_{n-2}} M_{n-2,n-1})^{\times k} \rar{\tau_k}
  & W(R_{n-1}; M_{n-1,0} \otimes_{R_0} \ldots \otimes_{R_{n-2}} M_{n-2,n-1})\\
\end{tikzcd}
}
\]
where the left vertical map is the cyclic permutation of order $nk$.
Under the ghost map, the isomorphism $T$ is compatible with the isomorphism
\[
  \prod_{k\geq 1} ((M_{0,1}\otimes_{R_1} \ldots \otimes_{R_{n-1}} M_{n-1,0})^{\circledcirc_{R_0} k})^{C_k} \to \prod_{k\geq 1} ((M_{n-1,0}\otimes_{R_0} \ldots \otimes_{R_{n-2}} M_{n-2,n-1})^{\circledcirc_{R_{n-1}} k})^{C_k}
\]
given on the $k$-th factor by the cyclic permutation of order $nk$.
The $n$-fold composition of $T$ defines an automorphism of
$
W(R_0; M_{0,1}\otimes_{R_1} \ldots \otimes_{R_{n-1}} M_{n-1,0})
$, which is the identity.
\end{prop}

\begin{proof}
Just as in the proof of Proposition  \ref{prop:weilaction}, uniqueness follows immediately since the images of the $\tau_k$ form a system of generators. Existence is checked in the case of suitably free $R_i$, $M_{i,i+1}$, such that the $\tlog$ is injective, by computing that the claimed action on ghost components acts correctly on elements of the form $\tlog(\tau_k(a_k))$. Finally, the statement about the $n$-fold iterate of this isomorphism also follows by observing that the corresponding map on ghosts is the identity.
\end{proof}

For any ring $R$, we let $\operatorname{Proj}_R$ denote the category of finitely generated projective right $R$-modules.

\begin{cor}
\label{cor:morita}
Every additive functor $A: \operatorname{Proj}_R \to \operatorname{Proj}_S$ induces a map of abelian groups $A_* : W(R) \to W(S)$ extending the functoriality of $W$ in ring homomorphisms. In particular Morita equivalent rings $R$ and $S$ have isomorphic Witt vectors $W(R) \cong W(S)$. 
\end{cor}
\begin{proof}
Any additive functor  $A : \operatorname{Proj}_R \to \operatorname{Proj}_S$ is of the form $A \cong (-) \otimes_R M$, where $M$ is the $R$-$S$-bimodule $M := A(R)$. Let $N$ be the $S$-$R$-bimodule $N := \Hom_S(M, S)$. There are bimodule maps
\[
\eta : R \to M \otimes_S N \qquad \text{and} \qquad \ev : N \otimes_R M \to S,
\]
where the second map is the evaluation, while the first map corresponds under the isomorphism 
\begin{equation}\label{iso_proj}
M \otimes_S N \cong M \otimes_S \Hom_S(M, S) \cong \Hom_S(M,M)
\end{equation}
to the map which sends $1 \in R$ to the identity. The isomorphism \eqref{iso_proj} uses the fact that $M$ is finitely generated projective over $S$. 
The desired map is defined as the composite
\[
W(R) = W(R;R) \stackrel{\eta_*}{\longrightarrow} W(R; M \otimes_S N) \cong W(S; N \otimes_R M) \stackrel{\ev_*}{\longrightarrow} W(S; S) = W(S),
\]
where the middle isomorphism is from Proposition \ref{prop:traceinvariance}.

If $R$ and $S$ are Morita equivalent we can find an $R$-$S$-bimodule $M$ such that $\eta$ and $\ev$ are isomorphisms, and it follows that the map above is also an isomorphism.
\end{proof}

\begin{rem}
In the $p$-typical case the Morita invariance of the Witt vectors has been shown by Hesselholt using a comparison to the topological invariant $\mathrm{TR}_0$. See \cite{HesselholtncW}, specifically (2.2.10) on page 130. He also mentions that ``One would like also to have an algebraic proof of this fact'' which is exactly what we have provided. It is remarkable that to prove this fact about Witt vectors of non-commutative rings one needs to introduce the more general notion of Witt vectors with coefficients. We consider this to be one of the main reasons to study this more general notion. 
\end{rem}

We finish this section by remarking that Corollary \ref{cor:morita} implies additional functoriality for the construction $R\mapsto W(R)$. 
\begin{enumerate}
\item
Every non-unital map of unital rings $f : R \to S$ gives rise to a functor  
\[
\operatorname{Proj}_R \to  \operatorname{Proj}_S \qquad P \mapsto P \otimes_R (f(1) \cdot S)
\]
 and thus to a map $W(R) \to W(S)$. One can of course see this directly, but Morita invariance gives a nice explanation for this additional functoriality.
\item
The functor $\oplus: \operatorname{Proj}_{R \times R} =  \operatorname{Proj}_R \times \operatorname{Proj}_R \to \operatorname{Proj}_R$ induces a map $W(R \times R) = W(R) \oplus W(R) \to W(R)$, which coincides with the group structure by an Eckmann-Hilton argument.
\item
For every map $R \to S$ such that $S$ is finitely generated projective over $R$ there is a `transfer' map $W(S) \to W(R)$ induced by the restriction functor 
$
\operatorname{Proj}_S \to  \operatorname{Proj}_R$. With some more work one can show that such a transfer map even exists if $S$ is a perfect complex over $R$.
\end{enumerate}

\subsection{Truncated Witt vectors with coefficients}\label{sec:truncated}

We recall that a subset $S\subseteq \N_{>0}$ is a \emph{truncation set} if it has the property that $ab\in S$ implies $a\in S$ and $b\in S$.

\begin{defn}
For a truncation set $S$ we define $W_S(R;M)$ to be the quotient of $W(R;M)$ by the closed subgroup generated by the elements $\tau_n(x)$ for all $n\not\in S$.
\end{defn}
For a prime $p$ the \emph{$p$-typical Witt vectors with coefficients} are defined as
\[
W_p(R;M) := W_{\{1,p,p^2,...\}}(R;M) 
\]
and for $n \geq 1$ the truncated version by
\[
W_{p,n}(R;M) := W_{\{1,p,p^2,\ldots ,p^{n-1}\}}(R;M) \ .
 \]
For every inclusion $S'\subseteq S$, we have a natural \emph{reduction map}
\[
R: W_S(R;M) \to W_{S'}(R;M).
\]
\begin{lemma}
For $S =\bigcup_i S_i$ an increasing union of truncation sets $\ldots \subseteq S_i \subseteq S_{i+1} \subseteq \ldots$, the map
\[
W_S(R;M) \to \varprojlim W_{S_i}(R;M)
\]
is an isomorphism.
\end{lemma}
\begin{proof}
Observe that the image filtration of the $\widehat{S}^{(n)}(R;M)$ on $W_S(R;M)$ is still Hausdorff and complete, by the same argument as in the proof of Lemma \ref{lem:complete}. 

The map $W_S(R;M) \to W_{S_i}(R;M)$ is surjective, with kernel $K_i$ topologically generated by the elements of the form $\tau_n(x)$ with $n\not\in S_i$, and since elements of the form $\tau_n(x)$ with $n\not\in S$ are already zero in $W_S(R;M)$, $K_i$ is actually generated by those $\tau_n(x)$ with $n\in S\setminus S_i$. 
We let $d_i$ be the minimal element of $S\setminus S_i$. So every element of $K_i$ has a representative of filtration $\geq d_i$. Since $\bigcup S_i = S$, $d_i$ tends to $\infty$ with $i$, and thus the $K_i$ also form a Hausdorff and complete filtration of $W_S(R;M)$, which implies the claim.
\end{proof}

For each truncation set $S$ we let $\pi_S$ be the projection map
$
\prod_{n\geq 1} (M^{\circledcirc_R n})^{C_n} \to \prod_{n\in S} (M^{\circledcirc_R n})^{C_n}.
$

\begin{lemma}\label{lemma:tlogS}
There exists a unique map $\tlog_S$ making the diagram
\[
\begin{tikzcd}
W(R;M) \rar{\tlog}\dar{R}& \prod_{n\geq 1} (M^{\circledcirc_R n})^{C_n}\dar{\pi_S}\\
W_S(R;M) \rar{\tlog_S} & \prod_{n\in S} (M^{\circledcirc_R n})^{C_n}
\end{tikzcd}
\]
commute. If the transfers $(M^{\circledcirc_R n})_{C_n}\to (M^{\circledcirc_R n})^{C_n}$ are injective for all $n\in S$, then $\tlog_S$ is also an embedding.
\end{lemma}
\begin{proof}
To check that $\tlog$ factors as claimed, it suffices to show that $\tlog(1-a_kt^k)$ for $k\not\in S$ is sent to $0$ under the projection map $\prod_{n\geq 1} (M^{\circledcirc_R n})^{C_n} \to \prod_{n\in S} (M^{\circledcirc_R n})^{C_n}$. Since
\[
\tlog(1-a_kt^k) = \sum_i \tr_{C_i}^{C_{ki}} a_k^i t^{ki},
\]
and $S$ contains no multiples of $ki$, this is clear.

Now assume that for each $n\in S$, the transfer $(M^{\circledcirc_R n})_{C_n}\to (M^{\circledcirc_R n})^{C_n}$ is injective. We want to show that $\tlog_S$ is injective. Let $x$ be an element in the kernel, say with a representative of the form $(1-a_kt^k + \ldots)$. If $k\not\in S$, we can factor this in the form $(1-a_kt^k) \cdot (1 - a_{k+1} t^{k+1} + \ldots)$, with the second factor still in the kernel of $\tlog_S$. If $k\in S$ on the other hand, then we have
\[
\tlog(1-a_kt^k + \ldots) = \tr_e^{C_k}a_kt^k + \ldots,
\]
  so $a_k$ lies in the kernel of the transfer $M^{\circledcirc_R k} \to (M^{\circledcirc_R k})^{C_k}$. By assumption this means that $a_k$ vanishes in $(M^{\circledcirc_R k})_{C_k}$. As in the proof of Lemma \ref{lem:tlogkernel}, this shows that we can multiply $(1-a_kt^k + \ldots)$ by a series with filtration $\geq k$ that vanishes in $W(R;M)$, in order to obtain a representative of the form $(1-a_{k+1}t^{k+1}+\ldots)$. This shows that any element which gets mapped by $\tlog_S$ to something of filtration $\geq k$ admits a representative of filtration $\geq k$. In particular, $\tlog_S$ is an embedding.equals a convergent product of elements trivial in $W_S(R;M)$, and thus vanishes.
\end{proof}

For a truncation set $S$, we define $S/n := \{k\in \N_{>0}\ |\ nk \in S\}$. This is again a truncation set.
\begin{prop}\label{prop:opS}
The Verschiebung and Frobenius maps descend to maps
\begin{align*}
&V_n: W_{S/n}(R;M^{\otimes_R n}) \to W_S(R;M),\\
&F_n: W_{S}(R;M) \to W_{S/n}(R;M^{\otimes_R n}),
\end{align*}
the Weyl action of $C_n$ on $W(R;M^{\otimes_R n})$ descends to a $C_n$ action on $W_S(R;M^{\otimes_R n})$, and the lax symmetric monoidal structure on $W(-;-)$ descends to one on $W_S(-;-)$. There are formulas for the ghost components of these maps analogous to the respective Propositions \ref{prop:verschiebung}, \ref{prop:frobenius}, \ref{prop:weilaction}, and \ref{prop:laxmonoidal}.
\end{prop}
\begin{proof}
The formulas given on ghost components for the various structure maps are all seen to be compatible with the projections onto the respective index sets. Now note that if $\tlog_S$ is injective, the kernel of $W(R;M) \to W_S(R;M)$ is the same as the preimage of the kernel of the projection map $\prod_{n\geq 1} (M^{\circledcirc_R n})^{C_n} \to \prod_{n\in S} (M^{\circledcirc_R n})^{C_n}$ under $\tlog$. So in the injective case, we see that the structure maps preserve these kernels and thus descend to structure maps on $W_S$. The statement for general pairs $(R;M)$ now follows by resolving by pairs where the relevant $\tlog$ are injective.
\end{proof}

 The following exact sequences are analogous to the sequences  of \cite[Lemma 3.2]{KaledinWpoly} for vector spaces over perfect fields of characteristic $p$.

\begin{prop}\label{sesRn}
Let $M$ be an $R$-bimodule, $S$ a truncation set and $k\geq 1$. We let $S'= S\setminus k\N$. Then there is a natural exact sequence
\[
W_{S/k}(R;M^{\otimes_R k})_{C_k}\stackrel{V_k}{\longrightarrow} W_{S}(R;M)\stackrel{R}{\longrightarrow} W_{S'}(R;M)\to 0.
\]
\end{prop}
\begin{proof}
Recall that $W_{S'}(R;M)$ is, by definition, the quotient of $W(R;M)$ by the closed subgroup generated by all $\tau_d(a_d)$ for $a_d \in M^{\times d}$ and $d\not\in S'$, i.e. $d=kl$. Equivalently, we can view this as the quotient of $W_S(R;M)$ by the image of that subgroup. We have to check that this coincides with the image of $V_k$. To see this, recall (Lemma \ref{lem:generators}) that $W_{S/k}(R;M^{\otimes_R k})$ is generated by elements of the form $\tau_l(a_{kl})$ with $a_{kl} \in M^{\times kl}$, and $l\in S/k$, or equivalently, $kl\in S$. Now observe that
\[
V_k(\tau_l(a_{kl})) = \tau_{kl}(a_{kl}),
\]
which proves the claim. 
\end{proof}

The  Verschiebung  is generally not injective. This is the case even for the usual noncommutative Witt vectors, that is when $M=R$, by \cite{HesselholtncWcorr}. The usual Witt vector Verschiebung is however injective if the ring has no torsion or if it is commutative.

\begin{prop}\label{Vinj}
The Verschiebung $V_k : W_{S/k}(R;M^{\otimes_R k})_{C_k}\to W_{S}(R;M)$ is injective when the transfers $(M^{\circledcirc_R n})_{C_n} \to (M^{\circledcirc_R n})^{C_n}$ are injective for every $n\in S$ with $k\mid n$. This is satisfied in particular if $(M^{\circledcirc_R n})_{C_n}$ has no $n$-torsion for each such $n$.
\end{prop}

\begin{proof} 
Assume $x\in W_{S/k}(R;M^{\otimes_R k})_{C_k}$ is in the kernel. Assume $x$ is not $0$, so there exists a maximal $l$ such that $x$ has filtration $\geq l$. We write $x = \tau_{l}(a_{kl}) + x'$ with $x'$ of filtration $\geq l+1$. If $kl\not\in S$, then $\tau_{l}(a_{kl}) = 0$ in $W_{S/k}$, and so $x=x'$ and $x$ has filtration $\geq l+1$, contradicting the maximality of $l$. So $kl \in S$.

  The leading term of $\tlog V_k(x)$ agrees with the one of $\tlog V_k(\tau_{l}(a_{kl})) = \tlog \tau_{kl}(a_{kl})$, which is given by $\tr_{e}^{C_{kl}}(a_{kl})$. Since $kl\in S$, the vanishing of $V_k(x)$ therefore implies that $\tr_{e}^{C_{kl}}(a_{kl})=0$. Since we assumed that the transfers $\tr: (M^{\circledcirc_R kl})_{C_{kl}} \to (M^{\circledcirc_R kl})^{C_{kl}}$ are injective, this implies that $a_{kl} = 0$ in $(M^{\circledcirc_R kl})_{C_{kl}}$. Similarly to the proof of Lemma \ref{lem:tlogkernel}, one can then write $(1+a_{kl}t^{l})$ in $\widehat{S}(R; M^{\otimes_R k})$ as a product of elements of the form $(1+x_i y_j t^l)\cdot(1+y_j x_i t^l)^{-1}$ with $x_i\in M^{\otimes_R i}$, $y_j\in M^{\otimes_R j}$, $i+j=kl$, and a remainder term of higher filtration. Observe that, from the definition of the $C_k$ action on $W(R; M^{\otimes k})$, the element $(1+x_i y_j t^l)(1+y_j x_i t^l)^{-1}$ represents $\tau_l(x_i y_j) - \sigma^j \tau_l(x_i y_j)$. It follows that $x\in W_{S/k}(R;M^{\otimes_R k})_{C_k}$ has filtration bigger than $l$, contradicting the maximality of $l$. Thus, $x=0$.
\end{proof}

\section{Characteristic Polynomials and cyclic $K$-theory}

In this section we define the characteristic polynomial for endomorphisms of finitely generated projective modules over non-commutative rings and compare it to Ranicki's and Sheiham's version of the Dieudonn\'e determinant, see \cite{Sheiham}. We will also discuss the group of rational Witt vectors and versions of the characteristic elements valued in this group.

\subsection{Characteristic polynomials for non-commutative rings} \label{Section: char polynomial}

We recall that for any ring $R$, not necessarily commutative, and any finitely generated projective right $R$-module $P$, the Hattori-Stallings trace is the additive map
\begin{equation}\label{map1}
\tr_R :  \End_R(P) \stackrel{\cong}{\longleftarrow} P\otimes_R P^\vee \xrightarrow{\ev}  R/[R,R] . 
\end{equation}
Here $P^\vee = \Hom_R(P, R)$, and the evaluation $P\otimes_R P^\vee \xrightarrow{\ev}  R/[R,R]$ is induced from the evaluation $P\otimes_{\mathbb{Z}} P^\vee \xrightarrow{\ev}  R$. It is only well-defined in the quotient $R/[R,R]$ by the additive subgroup $[R,R]\subseteq R$ generated by the commutators. The trace satisfies $\tr(AB) = \tr(BA)$ and thus descends to a map
$
\End_R(P)/[\End_R(P), \End_R(P)]  \longrightarrow R/[R,R] 
$
of abelian groups. The goal of this section is to give a (non-additive) refinement of the map \eqref{map1} through the first ghost component map $W(R) \to R/ [R,R]$, i.e. a map 
\[
\chi:  \End_R(P) \to W(R) \ .
\]
We first need an auxiliary construction. For every finitely generated projective $R$-module $P$ there is a  fully faithful functor
\begin{equation}\label{morita_functor}
(-)\otimes_{\End_R(P)}P:   \operatorname{Proj}_{\End_R(P)}\to \operatorname{Proj}_{R}
\end{equation}
and this induces by Corollary \ref{cor:morita} an additive map
\begin{equation}\label{map_end}
W(\End_R(P)) \longrightarrow W(R) \ .
\end{equation}
This map is not an isomorphism in general, but it is if the functor \eqref{morita_functor} is an equivalence of categories. By Morita theory this is the case if $P$ is free. We want to give a more general criterion for when this is the case.
\begin{defn}\label{defn_supported}
Let $R$ be a ring (not necessarily commutative) and $P$ a finitely generated, projective $R$-module. We say that $P$ is \emph{supported everywhere} if the functor $(-)\otimes_{\End_R(P)}P:   \operatorname{Proj}_{\End_R(P)}\to \operatorname{Proj}_{R}$ is an equivalence of categories.
\end{defn}

\begin{lemma}\label{lemma_support}
The module $P$ is supported everywhere precisely if the canonical evaluation map
\[
P^\vee \otimes_{\End_R(P)} P \to R
\]
is an isomorphism. If $R$ is commutative then this is also equivalent to the condition that $P$ has positive rank at every point of $\operatorname{Spec}(R)$. 

In general a sufficient condition for $P$ to have support everywhere is that $P$ contains $R$ as a summand. 
 \end{lemma}
\begin{proof}
The functor $(-)\otimes_{\End_R(P)}P:   \operatorname{Proj}_{\End_R(P)}\to \operatorname{Proj}_{R}$ has a right adjoint given by $(-)\otimes_R P^\vee$ and the counit of the adjunction is given by the map $M \otimes_R P^\vee \otimes_{\End_R(P)} P \to M$ which implies the criterion. For the second condition we note that we can check the first condition Zariski-locally, and for free modules it is equivalent to being non-trivial. 

Finally, to see that $P$ is supported everywhere, we have to show that $P$ is a generator of $\operatorname{Proj}_{R}$, which is immediate if $P$ contains a free summand. \end{proof}
\begin{prop}
\label{prop:moritaiso}
\leavevmode
\begin{enumerate}
\item\label{prop:item1}The map \eqref{map_end} is compatible with the product of traces on ghosts 
\[
\prod_{n\geq 1}\tr_R:  \prod_{n\geq 1} \End_R(P)/[\End_R(P), \End_R(P)] \longrightarrow \prod_{n\geq 1} R/[R,R].
\]
\item
The map \eqref{map_end} is compatible with direct sums in the sense that the maps  
\[
 W(\End_R P \times \End_R Q) \to W(\End_R(P \oplus Q)) \to W(R)
 \]
and 
\[
W(\End_R P \times \End_R Q)  \to W(\End_R P) \oplus W(\End_R Q) \to W(R) \oplus W(R) \xto{+} W(R)
\]
agree.
\item 
For a map of the form $\varphi: P \xto{\psi} R \xto{x} P$ the element $\psi(x) \in R$ is a representative of the class $\tr(\varphi) \in R/[R,R]$ and the map \eqref{map_end} 
sends $1 - \varphi t^n$ to 
$1 - \psi(x) t^n$.
\end{enumerate}
\end{prop}

\begin{proof}
We abbreviate $E :=  \End_R(P)$.
By definition and the proof of Corollary \ref{cor:morita}, the map \eqref{map_end} is the composite of the top row of the diagram
\[
\xymatrix@R=15pt@C=22pt{
W(E;E)\ar[r]_-{\cong}^-{\eta_*}\ar[d]_-{\tlog}
&W(E;P\otimes_RP^{\vee})\ar[r]_-{\cong}^-{T}\ar[d]_-{\tlog}
& W(R;P^\vee\otimes_{E} P)\ar[r]^-{\ev_*}\ar[d]_-{\tlog}
&W(R;R)\ar[d]_-{\tlog}
\\
\displaystyle\prod_{n\geq 1}(E^{\circledcirc_{E}n})^{C_n}\ar[r]_-{\prod \eta}^-{\cong}\ar[d]^-{\cong}
&\displaystyle\prod_{n\geq 1}((P\otimes_RP^{\vee})^{\circledcirc_{E}n})^{C_n}\ar[r]_-{\sigma}\ar[d]^-{\cong}
&\displaystyle\prod_{n\geq 1}((P^{\vee}\otimes_{E}P)^{\circledcirc_{R}n})^{C_n}\ar[r]_-{\prod\ev}
&\displaystyle\prod_{n\geq 1}(R^{\circledcirc_{R}n})^{C_n}\ar[d]^-{\cong}
\\
\displaystyle\prod_{n\geq 1}E/[E,E]\ar[r]_-{\prod \eta}^-{\cong}
&\displaystyle\prod_{n\geq 1}P\otimes_RP^{\vee}/[E,P\otimes_RP^{\vee}]\ar[rr]_-{\prod\ev}
&
&\displaystyle\prod_{n\geq 1}R/[R,R]
}
\]
where the composite of the bottom row is by definition the product of the traces.
The description in ghost components follows from the commutativity of this diagram. The three upper squares commute by the naturality of the ghost map and by Proposition \ref{prop:traceinvariance}. The lower middle isomorphism which makes the lower left square commute is defined by iterating the multiplication maps
\[
(P\otimes_RP^{\vee})\otimes_{\mathbb{Z}}(P\otimes_RP^{\vee})\xrightarrow{P\otimes_R\ev\otimes_RP^{\vee}}P\otimes_RR\otimes_R P^{\vee}\cong P\otimes_RP^{\vee}
\]
which correspond to the multiplication of $E$ under $\eta$. It is then easy to see that both composites in the lower right rectangle are
\[
(p_1\otimes_R\lambda_1)\otimes_E\dots \otimes_E(p_n\otimes_R\lambda_n)\longmapsto \lambda_n(p_1)\lambda_1(p_2)\dots \lambda_{n-1}(p_n).
\]
For the second property we note that both maps $W(\End_R P \times \End_R Q) \to W(R)$ are induced by functors 
\[
\operatorname{Proj}_{\End_R P \times \End_R Q} \to \operatorname{\Proj_R}
\]
which are easily seen to agree with the functor that sends $(M,N)$ to $(M \oplus N) \otimes_{(\End_R P \times \End_R Q)} (P \oplus Q) $. 

For the third property, consider $\varphi = x \circ \psi \in \End_R{P}$. Using property (2), we can assume that $P$ admits $R$ as a direct summand, by replacing it with $P\oplus R$ if necessary (and replacing $\varphi$ correspondingly by the map $\varphi\oplus 0: P\oplus R \to P\oplus R$). So we can choose $f: P\to R$ and $e: R\to P$ with $f\circ e=\id$. Now the element 
\[
\varphi\in \End_R P\cong P\otimes_R P^{\vee} \otimes_{\End_R P} \ldots \otimes_R P^\vee
\]
can be represented by the elementary tensor
\[
x\otimes_R f\otimes_{\End_R P} e \otimes_R \ldots \otimes_{\End_R P} e \otimes_R \psi.
\]
By Proposition \ref{prop:traceinvariance}, the map \eqref{map_end} therefore sends $(1-\varphi t^n)$ to $(1-\psi(x)f(e)\cdots f(e)t^n)=(1-\psi(x)t^n)$ as claimed. 
\end{proof}

\begin{rem}
Property (3) of Proposition \ref{prop:moritaiso} uniquely determines the map \eqref{map_end} in the sense that every other additive map $W(\End_R(P)) \longrightarrow W(R)$ with the same value on rank 1 endomorphisms agrees with our map \eqref{map_end}.
\end{rem}
\begin{defn}\label{char_el}
Let $P$ be a finitely generated projective $R$-module and $f\in \End_R(P)$.
We define the \emph{characteristic element}  $\chi_f \in W(R)$ to be the image of $f$ under the map
\[
\chi: \End_R(P)\stackrel{\tau}{\longrightarrow}W(\End_R(P)) \longrightarrow W(R)
\]
where  $\tau(f)=1 - f t$ as before and the second map is the map \eqref{map_end}.
\end{defn}

In the commutative case, where $W(R) = 1 + tR\pow{t}$, we will see in Proposition \ref{prop:chicomm} that on free $R$-modules \[\chi_A = \det(1 - A t),\] which agrees with the characteristic polynomial of $A$ up to a substitution. So one can view $\chi_f$ as a noncommutative generalisation of the characteristic polynomial. Note that $W(R)$ is in general a  quotient of the group of special units in the power series ring, so individual coefficients of $\chi_f$ are not well-defined. Also, for a general $R$ and $P$, there does not need to be a polynomial representative for $\chi_f$.

We now prove that the characteristic element $\chi_f$ satisfies the usual properties of the characteristic polynomial.

\begin{lemma}
\label{lem:charpolyproperties}
The characteristic element has the following properties:

\begin{enumerate}
\item It is natural under basechange.
\item For two endomorphisms  $f,g: P\to P$  of a finitely generated projective $R$-module, we have $\chi_{fg} = \chi_{gf}$.
\item For an endomorphism $f: P \to P$ the $n$-th ghost component of $\chi_f$ is given by the trace $\tr_R(f^n) \in R/[R,R]$. 
\item (Additivity) For a short exact sequence of endomorphisms, i.e. a commutative diagram in $\Proj_R$
\[
\begin{tikzcd}
0 \rar & P_1 \rar\dar{f_1} & P_2 \rar\dar{f_2} & P_3 \rar\dar{f_3} & 0\\
0 \rar & P_1 \rar & P_2 \rar & P_3 \rar & 0
\end{tikzcd}
\]
with exact rows, we have $\chi_{f_2} = \chi_{f_1}+\chi_{f_3}$ (cf. Remark \ref{rem:additive}).

\end{enumerate}
\end{lemma}
\begin{proof}
  The first statement is obvious from the definition.
The second statement is an immediate consequence of the fact that $1 - fg t = 1 - gft$ in $W(\End_R(P))$ by definition. The third statement follows immediately from Propositions \ref{prop:moritaiso}\eqref{prop:item1} and \ref{prop:tlogspecialvalues}.

For the fourth statement first consider the special case where the endomorphism splits, by which we mean that there exists a section $P_3 \to P_2$ such that, under the induced isomorphism $P_1\oplus P_3 \cong P_2$, $f_2$ corresponds to $f_1 \oplus f_3$.
In other words, $1- f_2t$ is the image of $1- (f_1, f_3)t $ under the map $\oplus :  W(\End_R(P_1) \times \End_R(P_3)) \to W(\End_R(P_2))$. Then the claim follows from (2) of Proposition \ref{prop:moritaiso}. 
For the general case we choose a section $s: P_3\to P_2$, and under the isomorphism $P_2 \cong P_1\oplus P_3$ we write $f_2$ as a ``block matrix''
\[
\left(\begin{array}{c|c}
f_1 &\rho\\
\hline
0 & f_3
\end{array}\right)
\]
where $f_i$ is an endomorphism of $P_i$, for $i=1,2$, and $\rho: P_3 \to P_1$ is $R$-linear.
In $W(\End_R(P_1\oplus P_3))$ we see that
\[
\tau\left(\begin{array}{c|c}
0 &\rho\\
\hline
0 & 0\end{array}\right)
= \tau\left(\left(\begin{array}{c|c}
1 &0\\
\hline
0 & 0\end{array}\right)\left(\begin{array}{c|c}
0 &\rho\\
\hline
0 & 0\end{array}\right)\right)
= \tau\left(\left(\begin{array}{c|c}
0 &\rho\\
\hline
0 & 0\end{array}\right)\left(\begin{array}{c|c}
1 &0\\
\hline
0 & 0\end{array}\right)\right)
= \tau\left(\begin{array}{c|c}
0 & 0\\
\hline
0 & 0\end{array}\right) = 1,
\]
using the relation $\tau(ab) = \tau(ba)$ that holds in $W(-)$ by definition. So the characteristic element of $\left(\begin{array}{c|c}
0 &\rho\\
\hline
0 & 0\end{array}\right)$ vanishes, and we further see that

\begin{gather*}
\left(1 - \left(\begin{array}{c|c}
f_1 &\rho\\
\hline
0 & f_3
\end{array}\right)t\right) = \left(1 - \left(\begin{array}{c|c}
0 &0\\
\hline
0 & f_3
\end{array}\right)t\right)\left(1 - \left(\begin{array}{c|c}
0 &\rho\\
\hline
0 & 0
\end{array}\right)t\right)\left(1 - \left(\begin{array}{c|c}
f_1 &0\\
\hline
0 & 0
\end{array}\right)t\right),\\
\left(1 - \left(\begin{array}{c|c}
0 &0\\
\hline
0 & f_3
\end{array}\right)t\right)\left(1 - \left(\begin{array}{c|c}
f_1 &0\\
\hline
0 & 0
\end{array}\right)t\right) = \left(1 - \left(\begin{array}{c|c}
f_1 &0\\
\hline
0 & f_3
\end{array}\right)t\right),
\end{gather*}
which together imply that
\[
\tau\left(\begin{array}{c|c}
f_1 &\rho\\
\hline
0 & f_3
\end{array}\right) = \tau\left(\begin{array}{c|c}
0 &0\\
\hline
0 & f_3
\end{array}\right)\tau\left(\begin{array}{c|c}
0&\rho\\
\hline
0 & 0
\end{array}\right)\tau\left(\begin{array}{c|c}
f_1 &0\\
\hline
0 & 0
\end{array}\right)=\tau\left(\begin{array}{c|c}
0 &0\\
\hline
0 & f_3
\end{array}\right)\tau\left(\begin{array}{c|c}
f_1 &0\\
\hline
0 & 0
\end{array}\right)=\tau\left(\begin{array}{c|c}
f_1 &0\\
\hline
0 & f_3
\end{array}\right)
\]
Thus, the characteristic element of $f_2$ agrees with the one of $f_1 \oplus f_3$, so by the previous case $\chi_{f_2} = \chi_{f_1} + \chi_{f_3}$.
\end{proof}

\begin{prop}\label{prop:chicomm}
For $R$ a commutative ring, $P=R^n$ a free module of rank $n$, and $f\in \End_R(P)$ an endomorphism, the characteristic element $\chi_f \in W(R) = 1 + tR\pow{t}$ is related to the classical characteristic polynomial $\chi_f^{\mathrm{cl}}$ by
\[
\chi_f(t) = \det(\id - tf) = t^n \det(t^{-1}\id - f) = t^n \chi_f^{\mathrm{cl}}(t^{-1}).
\]
\end{prop}
\begin{proof}
By naturality, it is sufficient to check the claim in the universal case, $R=\Z[a_{ij}\ |\ 1\leq i\leq n, 1\leq j\leq n]$, with $f$ the endomorphism given by the matrix $(a_{ij})$. Since $R$ is a domain, it embeds into an algebraically closed field $K$ of characteristic zero. As the map $W(R) \to W(K)$ is injective, it suffices to check that $\chi_f(t)$ and $t^n \chi_f^{\mathrm{cl}}(t^{-1})$ agree in $W(K)$. Over the algebraically closed $K$ however, $f$ can be brought into triangular form by a base change. This does not affect $\chi_f^{\mathrm{cl}}$ nor $\chi_f$, the latter because of part (2) of Lemma \ref{lem:charpolyproperties}. For triangular $f$, with diagonal entries $\lambda_1, \ldots, \lambda_n$, part (3) of Lemma \ref{lem:charpolyproperties} implies $\chi_f = \prod (1-\lambda_i t)$, which agrees with $t^n \prod (t^{-1} - \lambda_i) = t^n \chi_f^{\mathrm{cl}}$. 
\end{proof}

\begin{example}
  Even in the commutative case, the characteristic element is slightly more general than the  usual inverse characteristic polynomial (by which we mean $\det(\id-tf)$): it makes sense for non-free projective modules. We note however that in the commutative case our polynomial is given by the  formula
\[
\chi_f(t)=\sum_{i\geq 0}(-1)^i \tr(\Lambda^if)t^i
\]
which makes sense for projective modules. This is well-known and for example already appears in Almkvist's work \cite{Almkvist} and can be used as a definition. 

The usual (meaning: not inverse) characteristic polynomial can in the commutative situation also be extended to endomorphisms $f: P \to P$ of finitely generated projective modules over $R$. One simply defines it as before by the formula 
\[
\chi_f^{\mathrm{cl}}(t) = \det(t \cdot \id - f)
\]
 where $t \cdot \id - f$ is considered as an endomorphism of the $R[t]$-module $P[t]$. For this definition we use that the determinant makes sense for arbitrary endomorphisms $g: Q \to Q$ of finitely generated projective $S$-modules where $S$ is a commutative ring (here: $S = R[t]$ and $Q = P[t]$). One simply defines $\Lambda^{\mathrm{rk}} Q$ to be the top exterior power of $Q$, where $\mathrm{rk}: \operatorname{Spec} (S) \to \mathbb{N}$ is the locally constant rank function for $Q$. Then $\Lambda^{\mathrm{rk}} Q$ is a line bundle on $\mathrm{Spec}(S)$ and $ \det(g) := \Lambda^{\mathrm{rk}} g$ is an endomorphism of this line bundle, thus given by an element of $S$. Alternatively one finds a complement $Q'$ such that $Q \oplus Q'$ is free and defines the determinant of $g$ to be the determinant of the endomorphism $g \oplus \id_{Q'}$. We then have as in Proposition \ref{prop:chicomm} the relation
 \[
 \chi_f(t) = t^{\mathrm{rk}} \, \chi_f^{\mathrm{cl}}(t^{-1}).
 \]
and the coefficients of $\chi_f^{\mathrm{cl}}$ are given by traces of the exterior powers $\Lambda^{\mathrm{rk}-i}f$ similar to the formula above.

\end{example}

\begin{example}
We compute the characteristic element of the endomorphism
\[
A = \left(\begin{array}{cc} a&b\\c&d\end{array}\right) \in \Mat_{2\times 2}(R).
\]
The fourth statement of Lemma \ref{lem:charpolyproperties} shows that elements of $W(\Mat_{n\times n}(R))$ of the form $1-Nt$, with $N$ strictly lower or upper triangular, vanish. More generally, elements of the form $1 - N t^k = V_k(1 - Nt)$ vanish. So we can multiply an element of $W(\Mat_{n\times n}(R))$ with elementary matrices of the form $1-E_{ij}(t\lambda)$, $i\neq j$, $\lambda$ some power series, without changing it:
\begin{align*}
\tau(A) &= \left(\begin{array}{cc} 1-at&-bt\\-ct&1-dt\end{array}\right) \\
  &= \left(\begin{array}{cc} 1-at&-bt\\0&(1-dt)- c(1-at)^{-1}bt^2\end{array}\right)\\
  &= \left(\begin{array}{cc} 1-at&0\\0&(1-dt)- c(1-at)^{-1}bt^2\end{array}\right)
\end{align*}
From this we see that 
\[\chi_A = (1-at)(1-dt) - (1-at)c(1-at)^{-1}bt^2=
1-(a+d)t+(ad-cb)t^2-(ca-ac)bt^3-(ca^2-aca)bt^4-\dots.
\]
We observe that for commutative $R$, this simplifies to
\[
(1-at)(1-dt) - bc t^2 = 1 - (a+d)t + (ad-bc)t^2,
\]
which is, up to a substitution, the usual characteristic polynomial. However, as long as $a$ and $c$ do not commute, there is no reason to expect $\chi_A$ to have a polynomial representative.

Observe also that by a different row operation (killing the upper right entry), we could have obtained the representative
\[
\chi_A = (1-dt)(1-at) - (1-dt)b(1-dt)^{-1}ct^2, 
\]
which is not obviously equal to $(1-at)(1-dt) - (1-at)c(1-at)^{-1}bt^2$ under the relations we imposed on $W(R)$. It is interesting to see how the various symmetries of the characteristic polynomial arise in these noncommutative formulas.
\end{example}

\begin{rem}
Using the Gauss algorithm to compute the characteristic element as in the previous example also works for larger matrices. It can in fact be used to \emph{define} the characteristic element $\chi_f$, in which case the well-definedness is the crucial property to establish. This strategy is employed by Ranicki and Sheiham (see \cite{Sheiham2}) and we will return to this viewpoint in Remark \ref{Remark: non-commutative det} and \S\ref{sec_ratWitt}. 
\end{rem}

\begin{defn}\label{cyclictrace}
 The cyclic $K$-group $K_0^{\operatorname{cyc}}(R)$ is the quotient of the group completion of the abelian monoid of isomorphism classes $[P,f]$ of endomorphisms $f$ of finitely generated projective $R$-modules $P$, modulo the zero endomorphisms and the relation $[P_2,f_2]=[P_1,f_1]+[P_3,f_3]$ if $f_2$ is an extension of $f_1$ and $f_3$ as in Lemma  \ref{lem:charpolyproperties}.

The \emph{cyclic trace map} $K_0^{\operatorname{cyc}}(R) \to W(R)$ is the natural group homomorphism that sends an element $[P,f]$ to $\chi_f\in W(R)$. This is well-defined by Lemma \ref{lem:charpolyproperties}.
\end{defn}

\subsection{Determinants}

We now give a quick discussion of non-commutative determinants and Dieudonn\'e determinants over $R\pow{t}$. We let $R$ be a possibly non-commutative ring and consider a finitely generated projective $R$-module $P$. This gives rise to a finitely generated projective module $P\pow{t}$ over the power series ring $R\pow{t}$. We have \[
\End_{R\pow{t}}(P\pow{t}) = \End_R(P)\pow{t}.
\]
We let $\SEnd_{R\pow{t}}(P\pow{t})$ be the subset of $\End_{R\pow{t}}(P\pow{t})$ consisting of those endomorphisms which reduce to the identity modulo $t$. Under the isomorphism to $\End_R(P)\pow{t}$ these correspond to the power series whose first coefficient is the identity.

\begin{defn}
The  (reduced)  non-commutative determinant is the composite 
\[
\det:  \SEnd_{R\pow{t}}(P\pow{t}) \to W(\End_R(P)) \to W(R) 
\]
where the first map sends the 
 power series $\id - f t$ with $f \in \End_R(P)\pow{t}$ to the represented element in $W(\End_R(P))$ and the second map is the map \eqref{map_end}.
\end{defn}

\begin{rem}
With this determinant we can write the characteristic element of an endomorphism $f: P \to P$ (see Definition \ref{char_el}) as
\[
\chi_f = \det(\id - ft),
\]
where $\id - ft$ is considered as a special endomorphism of $P\pow{t}$.
\end{rem}

\begin{lemma}\label{lem_det_relations}
\leavevmode
The determinant is conjugation invariant, that is for every $R\pow{t}$-linear isomorphism $\alpha: P\pow{t} \to Q\pow{t}$ we have 
\[
\det(\alpha f \alpha^{-1}) = \det(f) \ .
\]
for any special endomorphism $f$ of $P\pow{t}$.
\end{lemma}
\begin{proof}
An isomorphism $P\pow{t}\to Q\pow{t}$ reduces to an isomorphism $P\to Q$. We can thus identify $P$ with $Q$ (note that the map \eqref{map_end} is clearly natural in isomorphisms of projective modules), and consider $\alpha$ as an automorphism of $P\pow{t}$.
Now the proof proceeds analogously to the proof of Lemma \ref{lem:charpolyproperties}. Note that by \ref{lem:otherrelations}, we have the relation
\[
(1+abt) = (1+bat)
\]
in $W(\End_R(P))$ for arbitrary, not necessarily homogeneous elements $a,b\in \End_R(P)\pow{t}$. Now if we write $f=1+gt$, we see
\[
\alpha(1+gt) \alpha^{-1} = (1 + \alpha g\alpha^{-1}t) = (1 + gt)
\]
in $W(\End_R(P))$.
\end{proof}

We now define determinants for special endomorphisms of an arbitrary finitely generated projective module $Q$ over $R\pow{t}$. 
\begin{lemma}\label{lem_iso}
Any finitely generated projective module $Q$ over $R\pow{t}$ is up to isomorphism of the form $P\pow{t}$ for $P$ a finitely generated, projective $R$-module.
\end{lemma}
\begin{proof}
We compare $Q$ with the module $P\pow{t}$ with $P = Q/t$. Using the fact that $Q$ is projective one finds a lift 
\[
\xymatrix{
& P\pow{t} \ar[d] \\
Q \ar[r] \ar@{-->}[ru] & P
}
\]
in the diagram which reduces to the identity modulo $t$. But since $Q$ and $P\pow{t}$ are both $t$-complete and $t$-torsion free (being f.g. proj.) this map has to be an isomorphism.
\end{proof}
  
For a $R\pow{t}$-module $Q$ we let $\SEnd_{R\pow{t}}(Q)$ be the subset of $\End_{R\pow{t}}(Q)$ consisting of all morphisms which reduce to the identity modulo $t$.  
  
\begin{defn}
For $Q$ a finitely generated projective $R\pow{t}$-module we define 
\[
\det: \SEnd_{R\pow{t}}(Q) \to W(R)
\]
by choosing an isomorphism $Q \cong P\pow{t}$ using Lemma \ref{lem_iso} and forming the composite 
\[
\SEnd_{R\pow{t}}(Q) \cong \SEnd_{R\pow{t}}(P\pow{t}) \to W(R) . 
\]
This does not depend on the choice of isomorphism by Lemma \ref{lem_det_relations}.
\end{defn}

\begin{lemma}\label{det_exact}
The determinant is additive in the following sense:  for a diagram 
\[
\begin{tikzcd}
0 \rar & Q_1\dar{f_1} \rar& Q_2 \rar\dar{f_2}& Q_3\rar\dar{f_3}& 0\\
0 \rar & Q_1 \rar& Q_2 \rar& Q_3 \rar& 0
\end{tikzcd}
\]
in which the vertical maps  
reduce to the identity modulo $t$ and the horizontal sequences agree and are exact, we have $\det(f_2) = \det(f_1) + \det(f_3)$.
\end{lemma}
\begin{proof}
By Lemma \ref{lem_iso} the diagram is up to isomorphism of the form
\[
\begin{tikzcd}
0 \rar & P_1\pow{t}\dar{f_1} \rar& P_1\pow{t}\oplus P_3\pow{t} \rar\dar{f_2}& P_3\pow{t} \rar\dar{f_3}& 0\\
0 \rar & P_1\pow{t} \rar& P_1\pow{t}\oplus P_3\pow{t} \rar& P_3\pow{t} \rar& 0
\end{tikzcd}
\]
where the horizontal sequences are given by inclusion and projection.
Now the proof proceeds as the proof of Lemma \ref{lem:charpolyproperties}, using that the relevant relations also hold for nonhomogeneous elements.
\end{proof}
\begin{rem} \label{Remark: non-commutative det}
There is a somewhat explicit form of this determinant, explained in \cite[Definition 14.3]{Ranicki} where it lands in a slightly different group of which our $W(R)$ is a quotient. First by adding the identity endomorphism on a complement to $f$ one can assume that $P = R^n$ is free so that $f$  is represented by a matrix  $M \in  \mathrm{Mat}_{n \times n}(R\pow{t})$. Modulo $t$ this matrix reduces to the identity. We can thus use Gaussian elimination to write $M$ as a product $M = LU$ with $L$ a lower triangular matrix with identity entries on the diagonal and $U = (u_{ij})$ an upper triangular matrix whose diagonal entries lie in $1 + t R\pow{t}$.  Then by Lemma \ref{det_exact} the determinant $\det(f) \in W(R)$ is represented by the product $u_{11} \cdot ... \cdot u_{nn}$ of the diagonal entries of $U$. 
\end{rem}

Now we claim that the determinant induces a map $\widetilde{K}_1(R\pow{t}) \to W(R)$ where $\widetilde{K}_1(R\pow{t}) = \ker(K_1(R\pow{t}) \to K_1(R))$ is the reduced $K_1$-group of $R\pow{t}$. Recall that  $K_1(R\pow{t})$ can be realized as
\[
K_1(R\pow{t}) = \frac{\langle [f]: Q \xrightarrow{\cong} Q \mid Q \text{ f.g. projective} \rangle}{[fg] = [f] + [g]\quad  [f \oplus g] = [f] + [g]} 
\]
where the generators are the isomorphism classes of automorphisms. By \cite[Lemma 4.1]{Sheiham2} the reduced group is isomorphic to
\[
\widetilde{K}_1(R\pow{t}) =  \frac{\langle [f]: Q \xrightarrow{\cong} Q \mid Q \text{ f.g. projective}, \varepsilon_*(f) = \id \rangle}{[fg] = [f] + [g],\quad  [f \oplus g] = [f] + [g]}
\]
where $\varepsilon: R\pow{t} \to R$ is the augmentation and the generators are again isomorphism classes, i.e. $f: Q \to Q$ and $g: Q' \to Q'$ are identified if there exists a commutative square of the form 
\[
\xymatrix{
Q \ar[r]^f \ar[d]_\alpha^\cong & Q \ar[d]_\alpha^\cong \\
Q' \ar[r]^g & Q' 
}
\]
for an $R\pow{t}$-linear isomorphism $\alpha$ (without any further condition on $\alpha$).\footnote{Note that Relation 3 in \cite[Lemma 4.1]{Sheiham2} is automatic since we take isomorphism classes of automorphisms.}
\begin{prop}\label{detK1toW} There is a well defined group homomorphism 
\begin{equation}\label{mappi}
\det : \widetilde{K}_1(R\pow{t}) \longrightarrow W(R)
\end{equation}
which sends an element represented by $\alpha \in \SEnd_{R\pow{t}}(Q)$ to $\det(\alpha)$ for any finitely generated projective $R\pow{t}$-module $Q$. 
\end{prop}
\begin{proof}
We only have to check that the map is well-defined. The first relation follows since $\det$ is a group homomorphism by definition and the second follows from Lemma \ref{det_exact}.
\end{proof}

There is also a group homomorphism $1 + t R \pow{t} \to \widetilde K_1(R\pow{t})$,
and the composite 
\[1 + t R \pow{t} \to \widetilde K_1(R\pow{t}) \to W(R)\] 
is the canonical quotient map. This shows that $W(R)$ is a quotient of $\widetilde{K}_1(R\pow{t})$, and if $R$ is commutative the projection is even an isomorphism. But in the non-commutative case this is not quite the case: The map  $1 + t R \pow{t} \to \widetilde K_1(R\pow{t})$ descends to an isomorphism
\begin{equation}\label{equation:sheiham}
\frac{(1 + tR\pow{t})}{ 1+pqt \sim 1+qpt} \stackrel{\cong}{\longrightarrow}\widetilde{K}_1(R\pow{t}),
\end{equation}
where $p$ and $q$ are arbitrary power series over $R$, see \cite{Pajitnov, PajRan} and also \cite[Theorem B and Proposition 3.4]{Sheiham2}). The left hand quotient looks similar to our Definition \ref{def_Witt} but it is not: The quotient here is purely algebraic and in Definition \ref{def_Witt} we close the subgroups by the topology.

\begin{rem}\label{rem:completion}
We see from the above discussion that the reason why the determinant map $\widetilde{K}_1(R\pow{t}) \longrightarrow W(R)$ is not an isomorphism is that the algebraically defined $K$-theory group does not take the the $t$-adic topology on the power series ring into consideration.  One can define a completed version 
\[
\widetilde{K}_1(R\pow{t})^\wedge  \cong \underleftarrow{\lim}_n \widetilde K_1(R[t]/t^n) 
\]
which is then isomorphic to the Witt vectors $W(R)$. Note that the completion $\widetilde{K}_1(R\pow{t}) \to \widetilde{K}_1(R\pow{t})^\wedge$ is surjective, thus the non-completeness of $\widetilde{K}_1(R\pow{t})$ lies entirely in the fact that it is in general not separated. 
\end{rem}
\subsection{Rational Witt vectors}\label{sec_ratWitt}

In this section we will construct a version of rational Witt vectors $W^{\rat}(R)$ mapping to $W(R)$ for a non-commutative ring $R$ and see that the characteristic polynomial actually takes values in $W^{\rat}(R)$. Most of the results are due to Sheiham (\cite{Sheiham, Sheiham2}) but we translate them into a language compatible with the current paper. Finally we discuss a generalisation of a theorem of Almkvist to the non-commutative setting.

The rough idea for rational Witt vectors is to replace the power series ring $R\pow{t}$ in the definition of $W(R)$ by the polynomial ring $R[t]$. There are several differences between these two rings, the most important one for us is that in the power series ring, an element $p(t) \in R\pow{t}$ is a unit precisely if the element $p(0) \in R$ is a unit. This of course fails for the polynomial ring and we will have to force it universally in the process of defining the rational Witt vectors, i.e. we will consider a certain localisation $L_\varepsilon R[t]$. We first introduce this localisation abstractly.

\begin{lemma}\label{local}
Let $\varepsilon: A \to R$ be a surjective map of not necessarily commutative rings. Then the following are equivalent.
\begin{enumerate}
\item Any endomorphism $f: Q \to Q$ of a finitely generated projective $A$-module $Q$, for which $\varepsilon_*(f)$ is an isomorphism of $R$-modules, is itself an isomorphism of $A$-modules;
\item Any element $a \in A$, for which $\varepsilon(a)$ is a unit in $R$, is itself a unit in $A$;
\item Any element $a \in A$, for which $\varepsilon(a) = 1$, is a unit in $A$;
\item The kernel of $\varepsilon$ is contained in the Jacobson radical of $A$.
\end{enumerate}
\end{lemma}
\begin{proof}
The implications $(1)  \Rightarrow (2) \Rightarrow (3) \Rightarrow (4)$ are clear. For $(4) \Rightarrow (1)$ we want to use the following version of Nakayama's lemma for non-commutative rings:
\begin{quote}
If a two-sided ideal $I \subseteq A$ is contained in the Jacobson radical and a finitely generated $A$-module $M$ is zero modulo $I$, then $M$ is zero. 
\end{quote}
Now assume $(4)$ holds and that $f: Q \to Q$ is a morphism as in $(1)$. We let $M$ be the cokernel of $f$ which is finitely generated and vanishes modulo $\ker(\varepsilon)$. Thus $M = 0$ and $f$ is surjective. Since $Q$ is projective we can choose a section $s$ of $f$. Then $\varepsilon_*(s)$ is also an isomorphism and repeating the argument for $s$ gives that $s$ is also surjective, thus $f$ an isomorphism. 
\end{proof}

Sheiham calls maps $A \to R$ as in  Lemma \ref{local} local maps. 
Note that the map $R\pow{t} \to R$ satisfies the equivalent conditions but $R[t] \to R$ does not. We want to form the universal `localisation' of $R[t]$ which does.

\begin{lemma}\label{Cohn_aug}
For every surjective map of rings $\varepsilon: A \to R$ there is an initial factorisation $A \to L_\varepsilon A \to R$ such that $L_\varepsilon A \to R$ is surjective and satisfies the conditions of Lemma \ref{local}.
\end{lemma}
\begin{proof}
We set $A_0 :=A$ and let $S_0 \subseteq A_0$ be the set of elements $s \in A_0$ such that $\varepsilon(s) = 1$. Then we form the localisation $A_1 := A_0[S_0^{-1}]$ and get a factorisation 
\[
A_0 \to A_1 \to R \ .
\]
Now $A_1 \to R$ does not necessarily satisfy the condition of Lemma \ref{local} since elements of the localisation $A_1$ might, in absence of any Ore condition, be arbitrary sums of words in $A$ and $S_0^{-1}$. We can repeat the procedure inductively to define $A_{n+1} := A_{n}[S_n^{-1}]$ with $S_n = \varepsilon_n^{-1}(1)$. This gives a tower
\[
A_0 \to A_1 \to A_2 \to \cdots
\]
of rings augmented over $R$ and we set $L_\varepsilon A := \colim A_i$. This ring has the desired properties by construction. 
\end{proof}
\begin{rem}
Ranicki and Sheiham use the Cohn localisation $\Sigma^{-1}A$ to construct $L_\varepsilon A$ where $\Sigma$ denotes the set of matrices over $A$ which become invertible under base-change to $R$. Then $\Sigma^{-1}A$ is the universal ring under $A$ over which the matrices in $\Sigma$ become isomorphisms. It turns out that, in contrast to our inductive procedure, a single iteration of this process is already enough to force the property that the morphism $\Sigma^{-1}A \to R$ satisfies the equivalent conditions of Lemma \ref{local}, see \cite[Section 3.1]{Sheiham}.
\end{rem}

For a ring $R$ we let $R[t] \to L_\varepsilon R[t]$ be the localisation of the polynomial ring $R[t]$ with its augmentation $\varepsilon= \ev_0$ as in Lemma \ref{Cohn_aug}.  Note that $t$ is still central in this ring and that for the map $\varepsilon: L_\varepsilon R[t] \to R$ the kernel is still generated by $t$ which follows from the fact that the short exact sequence 
\[
0\to R[t] \xto{t} R[t] \to R\to 0
\]
of $R[t]$-modules 
remains right exact after basechanging to $L_\varepsilon R[t]$ (and $R$ is, as an $R[t]$-module, already local, so $R\otimes_{R[t]} L_\varepsilon R[t]\cong R$).
\begin{defn}\label{def_ratWitt}
We define the rational Witt vectors of $R$ as the abelian group
\[
W^{\rat}(R) = \frac{(1+ t L_\varepsilon R[t] )^{\mathrm{ab}} }{1 + r p t \sim 1+ prt} 
\]
where $1+ t L_\varepsilon R[t] \subseteq L_\varepsilon R[t]$ has the group structure given by multiplication, and the relations run over all $r \in R$ and $p \in L_\varepsilon R[t]$. 
\end{defn}

\begin{rem}\leavevmode
\begin{enumerate}
\item By the definition of Witt vectors (and Lemma \ref{lem:otherrelations}(2)) there is a canonical map $W^{\rat}(R) \to W(R)$, which in the commutative case exhibits $W^{\rat}(R)$ as those power series in $R\pow{t}$ that can be written as a quotient of polynomials with constant term $1$. In the non-commutative case, however, it turns out that the map $W^{\rat}(R) \to W(R)$ is not necessarily injective as shown by Sheiham in \cite{Sheiham}. In general, the map still exhibits $W(R)$ as the completion of $W^{\rat}(R)$ with respect to the $t$-adic filtration.
\item
There are some slight variations of the relations that one can impose to get the same groups. For example one has
\[
W^{\rat}(R) = \frac{(1+ t L_\varepsilon R[t] ) }{1 +  pq t \sim 1+ qp t}
\]
    where $q$ and $p$ run through all elements in $L_\varepsilon R[t]$. Note that the commutators are contained in this subgroup (by \cite[Proposition 3.4]{Sheiham2}), so no abelianisation is needed. One also has
\[
W^{\rat}(R) = \frac{(1+ t L_\varepsilon R[t] ) }{1 +  pq \sim 1+ qp }
\]
where $p$ and $q$ run through all elements $L_\varepsilon R[t]$ such that $pq \in (t)$ and $qp \in (t)$. We have decided to  give the definition which is closest to our Definition \ref{def_Witt} of Witt vectors.

The fact that all these quotients agree can be seen as follows: 
    by \cite[Proposition 3.4]{Sheiham2}  the normal subgroup generated by $(1+pqt)(1+qpt)^{-1}$ for all $p,q\in L_\varepsilon R[t]$ agrees with the normal subgroup generated by all $(1+pq)(1+qp)^{-1}$ with $pq\in (t)$. We show that the normal subgroup generated by all $(1+pqt)(1+qpt)^{-1}$ agrees with the relations in Definition \ref{def_ratWitt}. It contains commutators (by \cite[Proposition 3.4]{Sheiham2}) and the elements where $p$ is homogeneous of degree $0$, so we only need to show that all $(1+pqt)(1+qpt)^{-1}$ are contained in the subgroup generated by commutators and the elements where $p$ is homogeneous of degree $0$, or equivalently, that in the abelianisation, the image of $(1+pqt)(1+qpt)^{-1}$ is contained in the subgroup generated by elements of the form $(1+p_0 q t)(1+q p_0 t)^{-1}$, with $p_0$ and $q_0$ homogeneous of degree $0$. To see this, observe that for $p$, $q$ with homogeneous degree $0$ components $p_0$, $q_0$:
 \begin{gather*}
 (1+pqt)(1-p_0 qt) = (1 + (p-p_0 - pqp_0 t)q t)\\
 (1+qpt)(1-qp_0 t) = (1 + q(p-p_0 - pqp_0 t)t),
 \end{gather*}
 and by Proposition 3.4(3) for $\zeta=0$ in \emph{loc. cit.}, the right hand sides agree in the abelianisation, so we have
 \[
 (1+pqt)(1+qpt)^{-1} = ((1-p_0 qt)(1-q p_0 t)^{-1})^{-1}
 \]
 in the abelianisation.
\end{enumerate}
\end{rem} 

The main reason to introduce rational Witt vectors here is that the characteristic element $\chi_f$ of an endomorphism $f: P \to P$, as defined in Definition \ref{char_el}, naturally lies in $W^{\rat}(R)$. More precisely we have the following result.
\begin{theorem}[Almkvist, Grayson, Ranicki, Sheiham]\label{thm_main_rat}
For every ring $R$ we have group isomorphisms 
\[
\xymatrix{
K_0^{\operatorname{cyc}}(R) \ar[r]^-\cong & \widetilde{K}_1(L_\varepsilon R[t]) \ar[r]_{\det}^\cong & W^{\rat}(R) 
}
\] 
where the first map sends a pair $[P,f]$ to the class of the automorphism $1 - ft: L_\varepsilon P[t] \to L_\varepsilon P[t]$ and the second map has the property that it sends the classes represented by elements in $1+ t L_\varepsilon R[t]$ (considered as an automorphism of the 1-dimensional module $L_\varepsilon R[t]$) to the class this element represents in $W^{\rat}(R)$.
\end{theorem}
\begin{proof}
For $R$ commutative the equivalences are due to Almkvist \cite{Almkvist} and Grayson \cite{Grayson}. In the non-commutative case they are shown in Ranicki \cite[Section 10,14]{Ranicki} and Sheiham \cite{Sheiham2}. 
\end{proof}

\begin{rem} 
We believe that there is a slightly incorrect definition in Ranicki's book: in Definition 14.7 and 14.10 of  \cite{Ranicki} the rational Witt vectors (of which our rational Witt vectors are a quotient) are defined as a subgroup of the Witt vectors. But by \cite{Sheiham} the inclusion is neither injective nor are the rational Witt vectors the group completion of $1 + t R[t]$. 
\end{rem}

Finally Sheiham \cite{Sheiham} defines the characteristic element as the composite
of the two maps of Theorem \ref{thm_main_rat}. This clearly maps to our characteristic element under the map $W^{\rat}(R) \to W(R)$. More generally, we can summarise the relation between the various characteristic polynomials and determinants defined in this paper and in \cite{Sheiham, Sheiham2} in the following commutative diagram:

\[
\xymatrix{
& K_0^{\operatorname{cyc}}(R) \ar[d]^-{\cong}\ar[rd]^-\cong & \\
(1+ t L_\varepsilon R[t] )\ar@{->>}[r]\ar[d] & \widetilde{K}_1(L_\varepsilon R[t] ) \ar[d] \ar[r]^\det_{\cong} & W^{\rat}(R)  \ar[d] \\
(1+t R\pow{t}) \ar@{->>}[r]  & \widetilde{K}_1(R \pow{t} )  \ar@{->>}[r]^{\det} & W(R)  
} 
\]
Here the composition  $K_0^{\operatorname{cyc}}(R) \to W(R)$ is the cyclic trace as defined in Definition \ref{cyclictrace}. The horizontal surjective maps in the diagram are in general not injective, but one can describe the kernels using the results of this section as well as \eqref{equation:sheiham} and Remark \ref{rem:completion}.
 In particular note that the lower determinant is not an isomorphism but rather a completion, as is the rightmost vertical map.

\subsection{Characteristic polynomials with coefficients}

In this section we will briefly explain how to generalize the constructions from \S\ref{Section: char polynomial}-\ref{sec_ratWitt} to a setting with coefficients in an $R$-bimodule $M$. The results are analogous to the results of the previous sections and we present them in the same order.

First  we treat the characteristic element. Let $P$ be a finitely generated projective right $R$-module and $M$ an $R$-bimodule. Then $\Hom_R(P, P\otimes_R M)$ is an $\End_R(P)$-bimodule isomorphic to $P\otimes_R M\otimes_R P^\vee$. We thus obtain a map
\[
\adjustbox{scale=0.98,center}{
\xymatrix@C=0pt{
W(\End_R(P); \Hom_R(P, P\otimes_R M))\ar@{}[r]|-{\cong}& W(\End_R(P); P\otimes_R M \otimes_R P^\vee)\ar@{}[r]|-{\cong}& W(R; P^\vee\otimes_{\End_R(P)} P \otimes_R M) \ar[d]^{(\ev\otimes M)_*}
\\
&& W(R;M)
}
}
\]
where the second isomorphism is from Proposition \ref{prop:traceinvariance}. The composite map is an isomorphism if $P$ is supported everywhere, e.g. if it admits a free non-trivial summand, see Lemma \ref{lemma_support}.

\begin{defn}\label{def_char_coefficients}
The characteristic element of an $R$-module map $f: P\to P\otimes_R M$ is  the image $\chi_f \in W(R;M)$ of $f$ under the map
\[
\chi :  \Hom_R(P, P\otimes_R M)\stackrel{\tau}{\longrightarrow}W(\End_R(P); \Hom_R(P, P\otimes_R M)) \longrightarrow W(R;M).
\]
\end{defn}

For $M=R$ this clearly reduces to Definition  \ref{char_el}. 

The analogue of  Lemma \ref{lem:charpolyproperties} also holds in this more general setting: Let $R,S$ be rings, $M$ an $R$-$S$-bimodule, $N$ a $S$-$R$-bimodule, $P$ a finitely generated projective $R$-module and $Q$ a finitely generated projective $S$-module. 
Given morphisms $f: P \to Q \otimes_S N$ and $g: Q \to P \otimes_R M$, we write $gf$ for the composite $P\to  Q \otimes_S N\to   P \otimes_R M \otimes_S N$ of $f$ and $g\otimes_SN$. 
\begin{lemma} \leavevmode
\label{lem_charelemwithcoeffproperties}
\begin{enumerate}
\item
In the situation above, the elements $\chi_{gf}$ and $\chi_{fg}$ correspond to each other under the trace isomorphism 
$W(R;  M\otimes_SN) \cong W(S; N \otimes_R M)$ of Proposition \ref{prop:traceinvariance}.
\item
Given a commutative diagram
\begin{equation}\label{diag_ext_coefficients}
\begin{tikzcd}
0 \rar& P_1 \rar\dar{f_1}& P_2 \rar\dar{f_2}& P_3 \rar\dar{f_3} \rar &  0\\
0 \rar& P_1 \otimes_R M \rar & P_2\otimes_R M \rar & P_3\otimes_R M \rar& 0
\end{tikzcd}
\end{equation}
with exact rows, we have $\chi_{f_2} = \chi_{f_1}+\chi_{f_3}$. 
\item
The $n$-th ghost component of the characteristic polynomial of $f: P\to P\otimes_RM$ is given by 
\[
\chi_f=\tr(f^n) \in M^{\circledcirc_R n}
\]
where the trace of the morphism $f^n: P \to P \otimes_R M^{\otimes_R n}$ is defined as the image of $f^n$ under the map
\[
\Hom_R(P, P \otimes_R  M^{\otimes_R n})  \stackrel{\cong}{\longleftarrow} P\otimes_R M^{\otimes_R n} \otimes_R P^\vee \xrightarrow{\ev}  M^{\otimes_R n} /[R,M^{\otimes_R n}] =  M^{\circledcirc_R n} \ .
\]
\end{enumerate}
\end{lemma}
\begin{proof}
  For the first statement, consider that the element represented by $\tau(g\otimes f)$ in 
\[
W(\End_R(P); P\otimes_R M\otimes_S Q^\vee\otimes_{\End_S(Q)} Q\otimes_S N\otimes_R P^\vee)
\] 
maps to the image of $gf$ in $W(\End_R(P); P\otimes_R M\otimes_S N\otimes_R P^\vee)$ and to the image of $fg$ in $W(\End_S(Q); Q\otimes_S N\otimes_R M \otimes_S Q^\vee)$ under suitable evaluation maps and trace isomorphisms. We obtain $\chi(gf)$ and $\chi(fg)$ by further application of trace isomorphisms and evaluation maps, and one easily obtains the claim from naturality of the trace isomorphisms.

The second proof proceeds analogously to Lemma \ref{lem:charpolyproperties}, and we also spell this out explicitly in the strictly more general situation of determinants, see Proposition \ref{prop_determinant_additivity} below. 

The third statement follows from unwinding the definitions, using the description of the ghost map on Teichm\"uller elements.
\end{proof}

Now we give the analogue for determinants with coefficients. We recall from Definition \ref{defn_tensoralgebra} that $\widehat{T}(R;M)$  denotes the completed tensor algebra of $M$ over $R$. We let $Q$ be a finitely generated, projective $\widehat{T}(R;M)$-module. Similarly to Lemma \ref{lem_iso} one shows that $Q$ is up to (non-canonical) isomorphism of the form $P \otimes_R  \widehat{T}(R;M)$ where $P$ is a finitely generated, projective $R$-module. More precisely $P$ is the basechange $\varepsilon_*Q$ along the augmentation map $\varepsilon: \widehat{T}(R;M) \to R$. 

\begin{defn}
A $\widehat{T}(R;M)$-linear endomorphism $f: Q \to Q$ is called special if $\varepsilon_*(f) = \id$. We denote the subset of those by 
\[
\SEnd_{\widehat{T}(R;M)}(Q) \subseteq \End_{\widehat{T}(R;M)}(Q) \ .
\]
\end{defn}

Now for a given automorphism $f: Q \to Q$ we choose an isomorphism $Q \cong P\otimes_R  \widehat{T}(R;M)$ and want to define the determinant of $f$ using this isomorphism. We shall then see that it is independent of the chosen isomorphism, similarly to the case without coefficients. We do this under the additional assumption that the module $P$ is supported everywhere, see Definition \ref{defn_supported}. This is not really a restriction as we can always replace $P$ by $P \oplus R$, which is supported everywhere by Lemma \ref{lemma_support} and we will later see that the determinant is a stable invariant.

\begin{lemma}
\label{lem_coefficients_isomorphisms}
Let $P$ a finitely generated projective right $R$-module and $M$ a $R$-bimodule. 
If $P$ is supported everywhere then we have canonical ring isomorphisms
\begin{gather*}
\End_{\widehat{T}(R;M)}(P\otimes_R \widehat{T}(R;M)) \cong \widehat{T}(\End_R(P); P\otimes_R M\otimes_R P^{\vee})\\
\SEnd_{\widehat{T}(R;M)}\left(P\otimes_R \widehat{T}(R;M)\right) \cong  \widehat{S}(\End_R(P);P \otimes_R M \otimes_R P^\vee)
\end{gather*}
\end{lemma}
\begin{proof}
For the first isomorphism, observe that 
\[
\End_{\widehat{T}(R;M)}(P\otimes_R \widehat{T}(R;M)) \cong \Hom_R(P,P\otimes_R \widehat{T}(R;M)) \cong \prod_{n\geq 0} P\otimes_R M^{\otimes_R n} \otimes_R P^{\vee},
\]
as $\End_R(P)$-bimodules. Since $P$ is supported everywhere, we can write $P\otimes_R M^{\otimes_R n}\otimes_R P^\vee \cong (P\otimes_R M\otimes_R P^{\vee})^{\otimes_{\End_R(P)} n}$, so we get
\[
\End_{\widehat{T}(R;M)}(P\otimes_R \widehat{T}(R;M)) \cong \widehat{T}(\End_R(P); P\otimes_R M\otimes_R P^{\vee}),
\]
and one sees directly that this isomorphism maps $\SEnd$ on the left isomorphically to $\widehat{S}$ on the right.
\end{proof}

\begin{defn}
For $P$ supported everywhere, we define the determinant as the composite
  \begin{align*}
    \det: \SEnd_{\widehat{T}(R;M)}\left(P\otimes_R \widehat{T}(R;M)\right) \cong{} & \widehat{S}(\End_R(P);P \otimes_R M \otimes_R P^\vee) \\
    \to{} & W(\End_R(P);P\otimes_R M\otimes_R P^\vee) \cong W(R; M),
  \end{align*}
where the last isomorphism is the trace property isomorphism from Proposition \ref{prop:traceinvariance}.
\end{defn}

\begin{lemma}
  \label{lem:det_conj}
  The determinant is conjugation invariant, i.e. for $P$ supported everywhere and any automorphism $\alpha: P\otimes_R \widehat{T}(R;M) \to P\otimes_R \widehat{T}(R;M) $ and any element $f\in \SEnd_{\widehat{T}(R;M)}$, we have $\det(\alpha f \alpha^{-1})=\det(f)$.
\end{lemma}
\begin{proof}
We can consider $\alpha$ as an element of $\widehat{T}(\End_R(P);P\otimes_R M \otimes_R P^\vee)$ and $f$ as an element of $\widehat{S}(\End_R(P); P\otimes_RM\otimes_R P^\vee)$ by Lemma \ref{lem_coefficients_isomorphisms}.

Now we proceed analogously to the proof of Lemma \ref{lem_det_relations}, writing $f=1+g$ with $g$ of positive filtration, and using the relation
$
(1 + \alpha g \alpha^{-1}) = (1 + g)
$.
\end{proof} 

Note that conjugation invariance allows us to obtain a well-defined notion of determinant for any special endomorphism $f: Q\to Q$ of a finitely generated projective $\widehat{T}(R;M)$-module $Q$, provided the base-change $Q\otimes_{\widehat{T}(R;M)} R$ is supported everywhere.

\begin{prop}
\label{prop_determinant_additivity}
For a short exact sequence of special endomorphisms of finitely generated, projective, everywhere supported $\widehat{T}(R;M)$-modules
\[
\begin{tikzcd}
0\rar & Q_1 \dar{f_1} \rar & Q_2 \dar{f_2}\rar & Q_3 \dar{f_3}\rar & 0\\
0\rar & Q_1 \rar & Q_2 \rar& Q_3\rar & 0
\end{tikzcd}
\]
we have that $\det(f_2) = \det(f_1) + \det(f_3)$.
\end{prop}
\begin{proof}
First, we can choose isomorphisms $Q_1 \cong P_1\otimes_R \widehat{T}(R;M)$ and  $Q_3 \cong P_3\otimes_R \widehat{T}(R;M)$ and split the exact sequence to reduce the situation to a diagram
\begin{equation}
\label{eq_diagram_coefficients_additivity}
\begin{tikzcd}[column sep=0.5cm]
0\rar & P_1\otimes_R \widehat{T}(R;M) \dar{f_1} \rar & P_1\otimes_R \widehat{T}(R;M)\oplus P_3\otimes_R \widehat{T}(R;M) \dar{f_2}\rar & P_3\otimes_R \widehat{T}(R;M) \dar{f_3}\rar & 0\\
0\rar & P_1\otimes_R \widehat{T}(R;M) \rar & P_1\otimes_R \widehat{T}(R;M)\oplus P_3\otimes_R \widehat{T}(R;M) \rar& P_3\otimes_R \widehat{T}(R;M)\rar & 0
\end{tikzcd}
\end{equation}
We can now consider $f_2$ as element of $\widehat{S}(\End_R(P_1\oplus P_3); (P_1 \oplus P_3) \otimes_R M \otimes_R (P_1 \oplus P_3)^\vee)$ using Lemma \ref{lem_coefficients_isomorphisms}, i.e. as an element of
\[
\prod_{n\geq 1} (P_1 \oplus P_3) \otimes_R M^{\otimes_R n} \otimes_R (P_1 \oplus P_3)^\vee.
\]
This splits additively into four factors of the form $\prod_{n\geq 1} P_i \otimes_R M^{\otimes_R n} \otimes_R P_j^\vee$, with $i,j\in \{1,3\}$, which one should regard as a block matrix decomposition of $f_2$ like in the proof of the additivity statement of Lemma \ref{lem:charpolyproperties}. Commutativity of the diagram \eqref{eq_diagram_coefficients_additivity} translates to the fact that the coordinates of $f_2$ in $\prod_{n\geq 1} P_i \otimes_R M^{\otimes_R n} \otimes_R P_i^\vee$ are $f_i$ for $i=0,3$, and that the coordinate of $f_2$  in $\prod_{n\geq 1} P_3 \otimes_R M^{\otimes_R n} \otimes_R P_1^\vee$ vanishes, i.e. $f_2$ is ``upper triangular''. Now we can use the same argument as in the proof of Lemma \ref{lem:charpolyproperties} to finish the proof, using the inhomogeneous relations from Lemma \ref{lem:otherrelations}.
\end{proof}

\begin{rem}
  Proposition \ref{prop_determinant_additivity} shows in particular that the determinant of a special endomorphism $f$ of $Q$ does not change if we stabilize it by passing to $f\oplus\id: Q\oplus \widehat{T}(R;M)\to Q\oplus \widehat{T}(R;M)$. So we can extend the definition to special endomorphisms of all finitely generated projective modules, not necessarily with the property that they are supported everywhere, preserving the properties from Lemma \ref{lem:det_conj} and Proposition \ref{prop_determinant_additivity}.
\end{rem}

Finally we want to discuss rational Witt vectors with coefficients. For an $R$-bimodule $M$ let 
\[
T(R; M) := \bigoplus_{n \geq 0} M^{\otimes_R n}
\]
 be the tensor algebra. This admits an augmentation $\varepsilon: T(R;M) \to R$ and we let $L_\varepsilon T(R;M) $ be the localisation as in Lemma \ref{Cohn_aug} which comes by definition with an augmentation $L_\varepsilon T(R;M) \to R$. We consider the subset
\[
S(R;M) \subseteq L_\varepsilon T(R;M)
\]
given by those elements in $L_\varepsilon T(R;M)$ which lie over $1 \in R$. 

\begin{defn}
The rational Witt vectors of $R$ with coefficients in $M$ are given by the abelian group
\[
W^{\rat}(R; M) := \frac{S(R;M)^{ab}}{1 + r p \sim 1+ pr} 
\]
where $r \in R$ and $p \in \ker(L_\varepsilon T(R;M) \to R)$.
\end{defn}
There is an obvious map $W^{\rat}(R;M) \to W(R;M)$ obtained from the map 
$S(R;M) \to \widehat{S}(R;M)$. The main result of \cite{Sheiham2} is that there is an isomorphism
\[
\det: \widetilde{K}_1(L_\varepsilon T(R;M) ) \xto{\cong} W^{\rat}(R;M) 
\]
induced by a determinant map. This map is defined by a similar strategy to the one explained in Remark \ref{Remark: non-commutative det}, i.e. by bringing matrices into upper triangular form using the Gauss-algorithm and then multiplying the diagonal entries. Part of the proof is to show that this is well-defined as an element of the rational Witt vectors.

There is also a version of the commutative diagram from the end of \S\ref{sec_ratWitt} with coefficients, as follows:
\[
\begin{tikzcd}
& K_0^{\operatorname{cyc}}(R;M) \dar\ar[rd] & \\
S(R;M)\rar[two heads]\dar & \widetilde{K}_1(L_\varepsilon T(R;M) ) \dar \rar["\det","\cong"'] & W^{\rat}(R;M)  \dar \\
\widehat{S}(R;M) \rar[two heads]  & \widetilde{K}_1(\widehat{T}(R;M) )  \rar["\det",two heads] & W(R;M)  
\end{tikzcd}  \ .
\]
Here $K_0^{\operatorname{cyc}}(R;M)$ refers to the group completion of the monoid of isomorphism classes of pairs $(P,f)$ with $P$ a finitely generated, projective $R$-module and $f$ an ``endomorphism with coefficients'' $f : P\to P\otimes_R M$, modulo pairs of the form $(P,0)$ and the relation that $[P_2,f_2] = [P_1,f_1] + [P_3,f_3]$ whenever we have an extension as in Lemma \ref{lem_charelemwithcoeffproperties}. The vertical map to $\widetilde{K}_1(L_\varepsilon T(R;M))$ is defined by sending $(P,f)$ to the automorphism $(1+f)$ of $P\otimes_R L_\varepsilon T(R;M)$, and the composite down to $W(R;M)$ can therefore be identified with the characteristic element $[P,f]\mapsto \chi_f$ (as in Definition \ref{def_char_coefficients}).

\begin{rem}\label{rem_rat_trace}
It will be a consequence of forthcoming work of the second and third author that the upper vertical map $K_0^{\operatorname{cyc}}(R;M) \to  \widetilde{K}_1(L_\varepsilon T(R;M))$  in the diagram is also an isomorphism and thus also the diagonal map  
$K_0^{\operatorname{cyc}}(R;M) \to W^{\rat}(R;M)$. Using this result one can deduce that $W^{\rat}(R; M)$ has the trace property, i.e. that there are isomorphisms
\begin{equation}\label{trace_rat}
W^{\rat}(R; M \otimes_S N) \cong W^{\rat}(S; N \otimes_R M)
\end{equation}
similar to the trace property for Witt vectors as shown in Proposition \ref{prop:traceinvariance}. 
This follows from the fact that $K_0^{\operatorname{cyc}}(R;M)$ has the trace property, which can be seen by applying some basic localisation sequences. However, we have not been able to construct the isomorphism \eqref{trace_rat} directly from the definition of the rational Witt vectors. 

If one assumes the trace property for rational Witt vectors then one can give more conceptual definitions of the determinant and the characteristic element valued in $W^{\rat}(R;M)$ similar to the constructions for non-rational  Witt vectors described at the beginning of the section.  

\end{rem}

\phantomsection\addcontentsline{toc}{section}{References} 
\bibliographystyle{amsalpha}
\bibliography{bib}

\noindent
\begin{tabular}{l}
Emanuele Dotto\\
Mathematical Institute, The University of Warwick\\
\textit{e-mail address:} \href{mailto:Emanuele.Dotto@warwick.ac.uk}{Emanuele.Dotto@warwick.ac.uk}
\end{tabular}
\vspace{.5cm}
\\
\begin{tabular}{l}
Achim Krause\\
Mathematisches Institut, Universit\"at M\"unster\\
\textit{e-mail address:}  \href{mailto:nikolaus@uni-muenster.de}{nikolaus@uni-muenster.de}
\end{tabular}
\vspace{.5cm}
\\
\begin{tabular}{l}
Thomas Nikolaus\\
Mathematisches Institut, Universit\"at M\"unster\\
\textit{e-mail address:}  \href{mailto:krauseac@uni-muenster.de}{krauseac@uni-muenster.de}
\end{tabular}
\vspace{.5cm}
\\
\begin{tabular}{l}
Irakli Patchkoria\\
Department of Mathematics, University of Aberdeen \\
\textit{e-mail address:} \href{mailto:irakli.patchkoria@abdn.ac.uk}{irakli.patchkoria@abdn.ac.uk}
\end{tabular}

\end{document}